\documentclass[12pt]{article}
\usepackage{amsmath,amsfonts,amsthm,amssymb,amscd}
\renewcommand{\Re}{\mathop{\rm Re}\nolimits}
\renewcommand{\Im}{\mathop{\rm Im}\nolimits}
\newcommand{\id}{\mathop{\rm id}\nolimits}
\newcommand{\diag}{\mathop{\rm diag}\nolimits}

\newcommand{\tr}{\mathop{\rm tr}\nolimits}
\newcommand{\Ker}{\mathop{\rm Ker}\nolimits}

\newcommand{\la}{\lambda}
\newcommand{\p}{\partial}
\newcommand{\e}{\varepsilon}
\newcommand{\C}{{\mathbb C}}
\newcommand{\R}{{\mathbb R}}

\newcommand{\Z}{{\mathbb Z}}

\newcommand{\T}{{\mathbb T}}
\newcommand{\N}{{\mathbb N}}

\newcommand{\OO}{{\cal O}}

\newcommand{\A}{{\mathfrak A}}

\newcommand{\const}{\mathop{\rm const}\nolimits}
\newcommand{\supp}{\mathop{\rm supp}\nolimits}
\newcommand{\linspan}{\mathop{\rm span}\nolimits}

\theoremstyle{plain}
\newtheorem{theorem}{Theorem}[section]
\newtheorem{lemma}[theorem]{Lemma}
\newtheorem{proposition}[theorem]{Proposition}

\theoremstyle{definition}
\newtheorem{definition}[{}]{Definition}
\theoremstyle{remark}

\def\om{\omega}
\def\a{\alpha}
\def\tri{\triangle}
\def\b{\beta}
\def\s{\sigma}

\numberwithin{equation}{section}

\begin{document}
\author{Sergei Kuksin and Galina Perelman}
\title{Vey theorem in infinite dimensions and \\
its application to KdV}
\date{}
\maketitle

\begin{abstract}
 We consider an integrable infinite-dimensional
Hamiltonian system in a Hilbert space
$H=\{u=(u_1^+,u_1^-; u_2^+,u_2^-;.\dots)\}$
with  integrals $I_1, I_2,\dots$ which can be written
as $I_j=\frac{1}{2}|F_j|^2$, where $F_j:H\rightarrow \R^2$,
$F_j(0)=0$ for $j=1,2,\dots$\,.
We assume that the maps $F_j$ define a germ of an analytic diffeomorphism
$F=(F_1,F_2,\dots):H\rightarrow H$,
such that $dF(0)=id$, $(F-id)$ is a $\kappa$-smoothing map ($\kappa\geq 0$)
and some other mild restrictions on $F$ hold.
Under these assumptions we show that the maps $F_j$
may be modified to maps $F_j^\prime$ such that
$F_j-F_j^\prime=O(|u|^2)$ and   each $\frac12|F'_j|^2$ still
 is an integral of motion.
Moreover, these maps jointly
define a germ  of an analytic
symplectomorphism $F^\prime: H\rightarrow H$,
 the germ $(F^\prime-id)$ is $\kappa$-smoothing, and each $I_j$ is
an analytic function of the vector $(\frac12|F'_j|^2,j\ge1)$.
 Next we show that the theorem with $\kappa=1$
applies to the KdV equation.
It implies that in the vicinity of the origin in a functional space
KdV admits the Birkhoff normal form and  the integrating
transformation has the form
`identity plus a 1-smoothing analytic map'.
\end{abstract}

\tableofcontents

\setcounter{section}{-1}

\section{Introduction}
\label{s0}

In his celebrated  paper \cite{V78} J.~Vey proved a local version of the
 Liouville-Arnold theorem which we
now state for the case of an  elliptic singular point.~\footnote{Vey's result applies as well to
 hyperbolic singular points and to singular points of mixed type.}
 Consider the standard symplectic linear space $(\R^{2n}_x,\omega_0)$,
 $\omega_0=\sum_{j=1}^ndx_j\wedge dx_{n+j}$.
 Let  $H(x)=O(|x|^2)$ be a germ of an analytic function
 \footnote{Here and everywhere below `a germ' means a germ at zero of
 a function or a map, defined in the vicinity
 of the origin.}
 and $V_H$ be the corresponding Hamiltonian vector field. It has a singularity
 at zero and we assume that in a suitable neighbourhood $\OO$ of the origin, $H$ has
 $n$ commuting analytic integrals  $H_1=H, H_2,\dots,H_n$ such that $H_j(x)=O(|x|^2)$
 for each $j$,
the quadratic forms $d^2H_j(0)$, $1\le j\le n$, are linearly independent
 and for all sufficiently small numbers
$\delta_1,\dots,\delta_n$ we have $\{x:\, H_j(x)=\delta_j\;\forall\,j\}\Subset \OO$.
Then in the vicinity of the origin exist symplectic   analytic coordinates
 $\{y_1,\dots,y_{2n}\}$ (i.e.
  $\sum_{j=1}^ndy_j\wedge dy_{n+j}=\omega_0$) such that   each hamiltonian $H_r(x)$
 may be written as $H_r(x)=\hat H_r(I_1,\dots,I_n)$,
$I_j=\frac12(y_j^2+y_{n+j}^2)$, where $\hat H_1,\dots,\hat H_n$ are germs of analytic
 functions on $\R^n$.

Vey's proof relies on the Artin theorem on a system of analytic equations, so it
 applies only to analytic finite-dimensional Hamiltonian systems. The theorem was developed
 and generalised in \cite{E84,  I89, E90, Z05}. In \cite{E84,  E90} Eliasson suggested a
 constructive proof of the theorem, which  applies both to smooth and analytic
  hamiltonians and may be generalised to infinite-dimensional
 systems. In this work we use  Eliasson's arguments to get an infinite-dimensional
 version of Vey's theorem, applicable to integrable Hamiltonian PDE.
Namely, we consider the $l_2$-space $h^0$, formed by sequences
 $u=(u_1^+,u_1^-, u_2^+, u_2^-,\dots)$,
provide it with the symplectic form $\omega_0=\sum_{j=1}^\infty du_j^+\wedge du^-_j$,
 and include $h^0$
in a scale $\{h^j,j\in\R\}$ of weighted $l_2$-spaces. Let us take any space $h^m$,
 $m\ge0$, and in a neighbourhood $\OO$ of the origin in $h^m$ consider commuting analytic
 hamiltonians $I_1,I_2,\dots$. We assume that $I_j=O(\|u\|^2_m)\ge0$ $\forall\,j$
 and that this system of functions is {\it regular} in the following sense: There are
 analytic maps $F_j:\OO\to \R^2,\  j\ge1$, such that $I_j=\frac12|F_j|^2$ and

i) the map $F=(F_1,F_2,\dots):\OO\to h^m$ is an analytic diffeomorphism on its image,

ii) $dF(0)=\,$id and the mapping $F-\,$id analytically maps $\OO\to h^{m+\kappa}$
 for some $\kappa\ge0$ (i.e., $F-\,$id is $\kappa$-smoothing). Moreover, for
 any $u\in\OO$ the linear operator $dF(u)^*-\,$id continuously maps $h^m$ to $h^{m+\kappa}$.

We also make some mild assumptions concerning Cauchy majorants for the
 maps $F-\,$id and $dF(u)^*-\,$id, see in Section~\ref{s1}. The main result
 of this work is the following theorem:

\begin{theorem}\label{tA}
Let the system of commuting analytic functions $I_1,I_2,\dots$ on $\OO\subset h^m$
 is regular. Then there are analytic maps $F'_j:\OO'\to\R^2$, defined on a suitable
 neighbourhood $0\in\OO'\subset \OO$, such that the map
 $F'=(F'_1,F'_2,\dots):\OO'\to h^m$ satisfies properties i), ii), it is a
 symplectomorphism, the functions $I'_j=\frac12|F'_j|^2$ commute and their
 joint level-sets define the same foliation of $\OO'$ as level-sets of the original
functions $I_j$. In particular, each $I_j$ is an analytic function of the variables
 $I'_1,I'_2,\dots$.
\end{theorem}

See Section~\ref{s1}  for a more detailed statement of the result and see
 Section~\ref{s4} for its proof. In Section~\ref{s3} we develop some infinite-dimensional
 techniques, needed for our arguments.

Theorem~\ref{tA} applies to study an integrable Hamiltonian PDE in the vicinity
 of an equilibrium. In Section~\ref{s2} we apply it to
 the KdV equation under  zero-meanvalue periodic
boundary conditions
\begin{equation}\label{3.2}
\dot u(t,x)=
\frac14u_{xxx}+6uu_x,\quad x\in S^1=\R/2\pi\Z,
\quad \int\limits_0^{2\pi}udx=0,
\end{equation}
 and to  the
 whole KdV hierarchy. The equations are regarded as Hamiltonian systems in a Sobolev space
 $H_0^m$, $m\ge0$, of zero-meanvalue functions on $S^1=\R/2\pi\Z$. The space
 is given the norm $\|u\|_m=|(-\Delta)^{m/2}u|_{L_2}$
and is
 equipped with the symplectic form $\nu$, where $\nu\big(u(\cdot),v(\cdot)\big)=
-\int_{S^1} (\p/\p x)^{-1}u(x)\cdot v(x)\,dx$. If $m\ge1$, then \eqref{3.2} is a Hamiltonian system
in $H^m_0$ with the analytic hamiltonian
$$
h_{KdV}(u)=\int \Big(-\frac18\, u_x^2+u^3\Big)\,dx.
$$

To apply Theorem~\ref{tA} we first normalise the symplectic form $\nu$ to
 the canonical form $\omega_0$. To do this we write any $u(x)\in H_0^m$ as
 Fourier series, $u(x)=\pi^{-1/2}\sum_{s=1}^\infty (u^+_s\cos sx - u^-_s\sin sx)$, and consider the map
$$
T:u(x)\mapsto v=(v_1^\pm,v_2^\pm,\dots),\qquad v_j^\pm=u_j^\pm j^{-1/2} \quad
\forall\,j.
$$
Then $T:H_0^m\to h^{m+1/2}$ is an isomorphism for any $m$, and $T^*\omega_0=\nu$.

The Lax operator for the KdV hierarchy is the Sturm-Liouville operator $L_u=-\p^2/\p x^2
-u(x)$. Let $\gamma_1, \gamma_2,\dots$ be the lengths of its spectral gaps. Then
$\gamma_j^2(u)$, $j\ge1$, are commuting analytic functionals which are integrals of
motion for all equations from the hierarchy. In \cite{Kap91} T.~Kappeler
 suggested a way to use the spectral theory of  the operator $L_u$ to construct
 germs of analytic maps $\Psi^j:h^{m+{1/2}}\to\R^2$, $j\ge1$, such that
 $\frac12|\Psi^j(v)|^2=\frac{\pi}{2j}
 \gamma_j^2(T^{-1}v)$.  In Sections~\ref{s5} we
show that the map $\Psi=(\Psi^1,\Psi^2,\dots)$
 meets assumptions i),~ii) with $\kappa=1$
(see Theorem~\ref{t3.1}). So the system of integrals 
$I_j(v)=\frac12 |\Psi^j(v)|^2$,
$j\ge1$, is regular.
 Accordingly,
 Theorem~\ref{tA} implies the following result (see Section~\ref{s2}):

\begin{theorem}\label{tB}
For any $m\ge0$ there exists a germ of an analytic symplectomorphism
$\overline\Psi:(H_0^m,\nu)\to(h^{m+1/2},\omega_0)$, $d\overline\Psi(0)=T$, such that

a) the germ $\overline\Psi-T$ defines a germ of an analytic mapping $H_0^m\to h^{m+3/2}$;

b) each $\gamma_j^2\,, j\ge1$, is an analytic function of the vector
 $\bar I= (\frac12|\overline\Psi^j(u)|^2, j\ge1)$. Similar, a hamiltonian of
 any equation from
the KdV hierarchy is an analytic function of $\bar I$ (provided that $m$ is
 so big that this hamiltonian is analytic on the space $H_0^m$);

c) the maps $\overline\Psi$, corresponding to different $m$, agree. That is, if $\overline\Psi_{m_j}$ corresponds to $m=m_j$,
$j=1,2$, then $\overline\Psi_{m_1}=\overline\Psi_{m_2}$ on $h^{\max(m_1,m_2)}$.
\end{theorem}

Moreover, Remark~4) to Theorem~\ref{t2.1} with $k=2$ and Remark at the end of Section~\ref{s5}
jointly imply that the map $\bar\Psi$ equals $\Psi\circ T$ up to $O(u^3)$:
\begin{equation}\label{000}
   \|\Psi\circ T(u)-\bar\Psi(u)\|_{h^{m+3/2}}\le {\rm const}\, \|u\|^3_m.
\end{equation}
In particular, $d^2\bar\Psi(0)=\psi_2\circ T$, where the map 
$v\mapsto \psi_2(v)$ is given by relations \eqref{5.20}. 
\medskip

Assertion b) of the theorem means that the map $\overline\Psi$ puts KdV (and other
equations from the KdV hierarchy) to the Birkhoff normal form.

In a number of publications, starting with \cite{Kap91}, T.~Kappeler with collaborators
 established existence of a global analytic symplectomorphism
$$
\Psi:(H_0^m,\nu)\to (h^{m+1/2},\omega_0),\quad d\Psi(0)=T,
$$
which satisfies assertion b) of Theorem~\ref{tB}, see in \cite{KaP}. Our work
 shows that a local version of Kappeler's result follows from  Vey's theorem. What
 is  more important, it specifies the result by stating that a local transformation
which integrates   the KdV hierarchy may be chosen  `1-smoother than its
linear part'. This specification is crucial to study qualitative properties of perturbed
 KdV equations, e.g. see \cite{KP09}.

A global symplectomorphism $\Psi$ as above integrates the KdV equation, i.e. puts it to the Birkhoff
normal form. Similar, the linearised KdV equation $\dot u=u_{xxx}$
may be integrated by the (weighted) Fourier transformation $T$.
An integrating transformation $\Psi$  {\it is not unique} . \footnote{in difference with the mapping
to the action variables $u\mapsto I\circ\Psi(u)$, which {\it is unique.}}
For the linearised
 KdV we do not see this ambiguity since $T$ is the only linear integrating
symplectomorphism. In the KdV case the best transformation $\Psi$ is the one which is
 the most close to the linear map $T=d\Psi(0)$ in the sense that the map $\Psi-T$ is the most smoothing.
 Motivated by Theorem~\ref{tB} and some other arguments (see in \cite{KP09}), we
are certain that there exists a (global) integrating symplectomorphism $\Psi$ such
that $\Psi-T$ is 1-smoother than $T$.
\footnote{We are cautious not to claim that the symplectomorphism $\Psi$,
 constructed in \cite{KaP}, possesses this extra smoothness since it is normalised
 by the condition
$$
{\rm if}\;\;  u(x)\equiv u(-x),\;\;  {\rm then}\;\;
 \Psi(u)=v=(v_j^\pm, j\ge1),\;\; {\rm where}\;\; v_j^-=0 \; \forall\,j.
$$
It is not obvious that an optimal global symplectomorphism  satisfies this
 condition, and we do not know if the local symplectomorphism $\overline\Psi$
 from Theorem~\ref{tB} meets it.
}

In Proposition~\ref{p1} we show that if a germ of an integrating  analytic
 transformation $\Psi$ is such that $\Psi-T$ is $\kappa$-smoothing, then $\kappa\le3/2$.
We conjecture that the 1-smoothing is optimal.

\medskip\par
\noindent{\bf Acknowledgment.} We  thank H.~Eliasson for discussion of the Vey
theorem.

\section{The main theorem.}\label{s1}
Consider a scale of Hilbert spaces $\{h^m,\,m\in \R\}$.
A space $h^m$ is formed by {\it complex} sequences $u=(u_j\in\C,\, j\geq 1)$
and is regarded as a {\it real} Hilbert space with the Hilbert norm
\begin{equation} \label{1.1}
\|u\|_m^2=\sum\limits_{j\geq 1}j^{2m}|u_j|^2.
\end{equation}
We will denote by $\left<\cdot,\cdot\right>$
the scalar product in $h^0$:
$\left<u,v\right>=\sum u_j\cdot v_j=\Re \sum u_j\bar v_j.$
For any linear operator $A: h^m\rightarrow h^n$ we will
denote by $A^*: h^{-n}\rightarrow h^{-m}$ the operator, conjugated to
$A$ with respect to this scalar product.
\par Below we study germs or real-analytic maps
\footnote{In Section \ref{s5} we mostly work with complex-analytic maps, so there {\it
analytic} stands for {\it complex-analytic}.}
$$ F: {\cal O}_\delta(h^m)\rightarrow h^n,\quad F(0)=0,$$
where  ${\cal O}_\delta(h^m)=\{u\in h^m \,\big |\,\|u\|_m<\delta\}$
and $\delta>0$ depends on $F$. Abusing language we will say that $F$ is an
 analytic germ $F:h^m\to h^n$.
 Any analytic  germ $F=(F^1,F^2,\dots)$ can be written as an absolutely and uniformly
convergent series
\begin{equation}\label{series}
F^j(u)=\sum\limits_{N=1}^\infty F^j_N(u),\quad
F^j_N(u)=\sum\limits_{|\alpha|+|\beta|=N}
A^j_{\alpha\beta}u^\a \bar u^\b,
\end{equation}
where $\a,\,\b \in \Z_+^\infty$,
$\Z_+=\N\cup\{0\}$.
We will write that $F(u)=O(u^l)$ if in \eqref{series} $F^j_N(u)=0$
for $N<l$ and all $j$.

Clearly,
\begin{equation*}
|F(u)|\leq \underline{F}(|u|),\quad
\underline{F}^j(|u|)= \sum\limits_{N=1}^\infty\sum\limits_{|\alpha|+|\beta|=N}
|A^j_{\alpha\beta}| |u|^{\a +\b}\le\infty.
\end{equation*}
Here $|F(u)|=(|F^1(u)|,|F^2(u)|,\dots)$, $|u|=(|u_1|,|u_2|,\dots)$
and $|u|^{\a +\b}=\prod|u_j|^{\a_j+\b_j}$. The inequality is understood component-wise.

\begin{definition}
An analytic germ $F$ as above is called normally
analytic (n.a.) if $\underline{F}$ defines
a germ of a real analytic map
$h_R^m\rightarrow h_R^n$,
where the space $h_R^m$ is formed by real sequences $(u_j)$,
given the norm \eqref{1.1}.
That is, each $N$-homogeneous map 
${\underline F}_N^j(v)=\sum\limits_{|\alpha|+|\beta|=N}
|A^j_{\alpha\beta}| v^{\a +\b}$, where  $v\in h_R^m$,
satisfies
$\|{\underline F}_N(v)\|_n\leq C R^N\|v\|_m^N\ $
for suitable $C,\,R>0$.
\end{definition}

Take any $m\geq 0$ and $\kappa\geq 0$.
\begin{definition}
A n.a. germ
 $F: h^{m}\rightarrow  h^{m+\kappa}$
belongs to  $\A_{m,\kappa}$
if $F=O(u^2)$ and the adjoint map $dF(u)^*v$ is such that
\begin{equation}\label{1.5}
\underline{dF(|u|)^*|v|}=\Phi(|u|)|v|.
\end{equation}
Here the linear
 map $\Phi(|u|)=\Phi_F(|u|)\in {\cal L}(h_R^m,h_R^{m+\kappa})$
has non-negative matrix elements and defines an analytic germ
$|u|\mapsto\Phi(|u|)$, $h_R^m\rightarrow {\cal L}(h_R^m,h_R^{m+\kappa})$.
\end{definition}

The notion of a n.a. germ formalizes the method of Cauchy majorants in a way,
convenient for our purposes. We study the class of n.a. germs and its subclass
$\A_{m,\kappa}$ in Section~\ref{s3}.

We will write elements of the spaces $h^m$
as $u=(u_k\in \C,\,\, k\geq 1)$,
$u_k=u_k^++iu_k^-$, $u_k^\pm \in \R$, and provide  $h^m$,
$m\geq 0$, with a symplectic structure by means of the two-form
$\omega_0=\sum du_k^+\wedge  du_k^-$. This form may be written as
$\omega_0=idu\wedge du$.
 Here and below for any antisymmetric (in $h^0$) operator $J$
we denote by
$Jdu\wedge du$ the 2-form
\begin{equation}\label{2.0}
(Jdu\wedge du)(\xi,\eta)=<J\xi,\eta>.
\end{equation}
The form $\omega_0$ is exact, $\omega_0=d\alpha_0$,
where
$$\a_0=\frac{1}{2}\sum u_k^+du_k^--\frac{1}{2}\sum u_k^-du_k^+
=\frac{1}{2}(iu)du.$$
For a map $f: h^m\rightarrow h^{-m}$,
$f(u)du$ stands for the  one-form
$$(f(u)du)(\xi)=\sum\limits_{j=1}^\infty f_j(u)\cdot \xi_j
=\sum\limits_{j=1}^\infty \Re f_j(u) \bar\xi_j.$$
By $\{H_1,H_2\}$ we will denote the Poisson brackets of
functionals $H_1$ and $H_2$, corresponding to $\om_0$:
$\
\{H_1,H_2\}(u)=\left<i\nabla H_1(u),\nabla H_2(u)\right>.
$
Functionals $H_1$ and $H_2$  {\it commute}  if $\{H_1,H_2\}=0$.

\begin{theorem}\label{t2.1}
Assume that for some $m\geq 0$
there exists a real analytic germ
$\Psi: h^m\rightarrow h^m$ such that
\begin{itemize}
\item[i)] $d\Psi(0)=id$ and $(\Psi-id)\in \A_{m,\kappa}$ for some    $\kappa\geq 0$;
\item[ii)] the functionals $I^j(\Psi(u))=\frac{1}{2}\left|\Psi^j(u)\right|^2$,
 $j\ge1$, commute with each other.
\end{itemize}
Then there exists a germ $\Psi^+: h^m\rightarrow h^m$ 
 which satisfies i), ii) with the same $\kappa$, and such that
\begin{itemize}
\item[a)] foliation of the vicinity of the origin in $h^m$
by the sets
\begin{equation}\label{fol}
    \{\left|\Psi^j\right|^2=const_j,\,\forall j\};
\end{equation}
is the same as by the sets $\{\big|{\Psi^+}^j\big|^2=const_j,\,\forall j\}$.
\item[b)] the germ $\Psi^+$ is symplectic:
${\Psi^+}^*\om_0=\om_0$.
\end{itemize}
\end{theorem}
The theorem is proved in Section~\ref{s4}.
\smallskip

\noindent
{\it Remarks.} 1) The sets, forming the foliation \eqref{fol}, are tori of
 dimension $\#\{const_j>0\}$, which is $\le\infty$.

2) By the item a) of the theorem each $I^j(\Psi(u))$ is a function of the vector
$I^+=\{I^{+j}=\frac12|\Psi^{+j}|^2, j\ge1\}$. In fact, $I^j$ is an analytic
 function of $I^+$ with respect to the norm $\|I^+\|=\sum |I^{+j}|j^{2m}$. E.g.,
 see the proof of Lemma~3.1 in \cite{K2}.

3) The map $\Psi^+$ is obtained from $\Psi$ in a constructive way, independent from $m$.

4) The form $\omega_1=(\Psi^*)^{-1}\omega_0$ equals $\omega_0$ at the origin. So
$\omega_\Delta(u):=\omega_1(u)-\omega_0(u)=O(u)$. Assume that $\omega_\Delta=O(u^k)$
with some $k\ge2$. Then a strightforward analysis of the proof of Theorem~\ref{t2.1} 
shows that  
$\ \|\Psi(u)-\Psi^+(u)\|_{m+\kappa}\le{\rm const}\|u\|_m^{k+1}.
$

5) The theorem above is an infinite-dimensional version of Theorem~C in \cite{E90}
which is the second step in Eliasson's proof of the Vey theorem. At the first step he
{\it proves} that any $n$ commuting integrals $H_1,\dots,H_n$ as in Introduction can
be written in the form ii). In difference with his work we have to {\it assume} that the integrals are of the form ii), where the maps $\Psi_1,\Psi_2,\dots$ have additional
properties, specified in i). Fortunately, we can check i) and ii) for some important
infinite-dimensional systems.

\section{Application to the KdV equation}\label{s2}
\par To apply Theorem \ref{t2.1}
we need a way to construct germs of analytic maps $\Psi: h^m\rightarrow h^m$
which satisfy i) and ii).
Examples of such maps may be obtained from
Lax-integrable Hamiltonian  PDEs
\begin{equation}\label{3.1}
\dot u(t)=i\nabla H(u),\quad u(t)\in h^m.
\end{equation}
(We normalized  original Hamiltonian PDEs and
wrote them as the Hamiltonian systems \eqref{3.1} in the symplectic space as in Section~\ref{s1}).
The Lax operator $L_u$, corresponding
to equation \eqref{3.1},
is such that its spectrum $\sigma(L_u)$ is an integral of motion
for \eqref{3.1}.
Spectral characteristics of $L_u$
 may be used to construct
(real)-analytic germs $\Psi^j:h^m\rightarrow \R^2\simeq\C$
such that the functions
$\frac{1}{2}|\Psi^j|^2$, $j\geq 1$,
are functionally independent integrals of motion.
For some integrable  equations these germs jointly define
a germ of an analytic diffeomorphism
 $u\mapsto \Psi=(\Psi^1,\Psi^2,\dots)$, satisfying i) and ii).
Below we show that this is the case for the KdV equation.
Our construction is general and directly applies to some
other integrable equations
(e.g. to the defocusing Schr\"odinger equation).

 Consider the KdV equation \eqref{3.2}.
This is a Hamiltonian equation in any Sobolev space $H_0^m,\ m\ge1$, given
symplectic structure by the form $\nu$, see Introduction. It is Lax-integrable
 with the Lax operator $L_u=-\p^2/\p x^2-u(x)$. Let $\gamma_1(u), \gamma_2(u), \dots$
 be the sizes of spectral gaps of $L_u$ (e.g., see in \cite{K2, KaP}). It is
 well known that $\gamma_1^2(u), \gamma_2^2(u),\dots$ are commuting analytic
 integrals of motion for \eqref{3.2}, as well as for other equations from the KdV
 hierarchy, see in \cite{KaP}.
 \medskip

In Section~\ref{s5} we show that the spectral theory of $L_u$ may be used to construct an analytic germ $\Psi:h^{1/2}\to h^{1/2}, \Psi=(\Psi^1,\Psi^2,\dots), \Psi^j\in\R^2$, with
 the following properties:

\begin{theorem}\label{t3.1}
For any $m'\ge1/2$, $\ \Psi$ defines a real-analytic germ $\Psi:h^{m'}\to h^{m'}$ such that

i) $d\Psi(0)=\id$ and $(\Psi-\id) \in\frak A_{m',1}$;

ii) for any $j\ge1$ and $v\in h^{m'}$ we have $\frac12|\Psi^j(v)|^2=
\frac{\pi}{2j}\gamma_j(u)^2$, where
$u(x)=\frac1{\sqrt\pi}\Re \sum_{j=1}^\infty\sqrt j\, v_je^{ijx}$.
\end{theorem}

 Applying Theorems~\ref{t3.1} and~\ref{t2.1} to the KdV equation, written in the variables 
 $v=T(u)\in h^{m'}$,
 we get Theorem~\ref{tB}, stated in the Introduction. Indeed, assertions a) and b) follow from the two theorems and Remark~2 to Theorem~\ref{t2.1} since the hamiltonian of any $n$-th KdV is a function of the lengths of spectral gaps. Assertion c) follows from Remark~3.
\smallskip

 Towards the optimality of Theorems~\ref{tB} and  \ref{t3.1}
we have the following partial results.
\begin{proposition}\label{p1}
Assume that there exists a real-analytic germ
$\Psi: H_0^m\rightarrow h^{m+1/2}$ $\,\forall\,\, m\geq0$,
$d\Psi(0)=T$, such that:

a) for each $m\ge0$, $\Psi-T$ defines a germ of analytic mapping
$H_0^m\rightarrow h^{m+1/2+\kappa}$ with some $\kappa\geq 0$;

b) the hamiltonian $h_{KdV}$  of the KdV equation is a function of the variables
$\frac12\big|\Psi^j(u)\big|^2$, $j\geq 1$, only.

Then $\kappa\leq 3/2$.
\end{proposition}
\begin{proof}
We may assume that $\kappa\ge1$.
Denote by $G$ the germ $G=\Psi^{-1}\circ T: H_0^m\rightarrow H_0^m$.
We have  $dG(0)=\id$ and
$G-\id: H_0^m\rightarrow H_0^{m+\kappa}$. So $G(u)=u+\sum_{N=2}^\infty G_N(u)$,
where
\begin{equation}\label{a1}
\|G_N(u)\|_{H^{m+\kappa}}\leq C^N \|u\|^N_{H^m} \quad \forall\,\, N\geq 2,
\end{equation}
for each $m\ge0$, with some $C=C(m)$.
 Consider the functional
$K=h_{KdV}\circ G$.  It  defines a germ of analytic mapping
$H_0^1\rightarrow \R$ and
 can be written as an absolutely and uniformly
convergent series
$\
K(u)=\sum_{n=2}^\infty K_n(u),
$
where $K_n(\cdot)$ is an $n$-homogeneous functional
on $H_0^1$. Then
$$ K_2(u)=-\frac{1}{8}\int u_x^2\,dx\,, \quad
K_3(u)=\int\big(-\frac14u_x\p_xG_2(u)+u^3\big)\,dx\,.
$$
 It follows from assumption b) that
$K_{2l+1}$, $l=1,2,\dots,$ vanish identically.
In particular, $K_3(u)\equiv 0$.  Together with \eqref{a1} this leads to the
relations
\begin{equation*}
\bigg|\int u^3dx\bigg| =\frac14\bigg|\int u_x\p_xG_2(u)\,dx\bigg|
\leq C\|u\|_{H_0^{2-\kappa}} \|G_2(u)\|_{H_0^{\kappa}}
\leq C\|u\|_{H_0^{2-\kappa}} |u|^2_{L_2},
\end{equation*}
valid for each $u\in H_0^1$. If $\kappa\ge2$ we have an obvious contradiction. It remains to consider
the case when $1\le\kappa<2$. Now $\|u\|_{H_0^{2-\kappa}}
 \le \|u\|_{H_0^{1}} ^{2-\kappa}|u|_{L_2}^{\kappa-1}$
 and the inequality above implies that
\begin{equation}\label{a2}
\bigg|\int u^3dx\bigg|
\leq C\|u\|_{H_0^{1}} ^{2-\kappa}   |u|^{1+\kappa}_{L_2}
\qquad \forall \,
u\in H_0^1.
\end{equation}
For $0<\e\le1$ we define $v_\e(x)$ as the continuous piece-wise linear $2\pi$-periodic function, equal
$\e^{-2}\max(\e-|x|,0)$ for $|x|\le4$. Then $u_\e:=v_\e-(2\pi)^{-1}  \in H_0^1$  and
$$
\int u_\e^3\,dx\sim \e^{-2},\quad \int u_\e^2\,dx\sim \e^{-1},\quad
\int \left(\frac{\p}{\p x}\, u_\e\right)^2dx\sim \e^{-3}.
$$
Substituting $u_\e$ in \eqref{a2} we get that
$\e^{-2}\le\const\, \e^{-\frac32(2-\kappa)}\e^{-\frac12(1+\kappa)}
$
for each $\e$. So  $\kappa\le\frac32$, as stated.
\end{proof}

If a germ $\Psi: H_0^m\rightarrow h^{m+1/2}$ is defined for a single value of $m$ we have
a weaker result:

\begin{proposition}\label{p2} Let for some $m'\ge1$ there exists a germ of
  real-analytic symplectomorphism
$\Psi: (H_0^{m'}, \nu)
\rightarrow (h^{m'+1/2},\omega_0)
$,
$d\Psi(0)=T$, satisfying a) and b) in Proposition~\ref{a1} with $m=m'$. Then
$\kappa\le2$.
\end{proposition}
\begin{proof}
Assume that $\kappa>2$. Keeping the notations above we still have $K_3=0$. So
\begin{equation*}
0=\nabla K_3(u)=\frac14\Delta G_2(u)+\frac14dG_2(u)^*\Delta u+3u^2\,.
\end{equation*}
The first term in the r.h.s. clearly belongs to  $H^{m'+\kappa-2}$.
The germ $G$ is a symplectomorphism of $(H_0^{m'},\nu)$.  Therefore
$dG(u)^*JdG(u)\equiv J$, $J=-(\p/\p x)^{-1}$, and $dG(u)^*$ maps $H_0^{m'+1}$ to itself. Since
$dG(u)^*$ also maps to itself $H_0^{-m'}$, then by interpolation $dG(u)^*:H_0^s\to H_0^s$ for each $s$ in
$[-m',m'+1]$. So the second
  term in the r.h.s. also belong to $H^{m'+\kappa-2}$. Since $\kappa-2>0$, then the sum of the first two terms
  cannot cancel identically the third, belonging to $H_0^{m'}$. Contradiction.
\end{proof}

\section{Properties of normally analytic germs}\label{s3}

\begin{lemma}
\label{l1.1}
If $F: h^{n_1}\rightarrow  h^{n_2}$
and $G: h^{n_2}\rightarrow  h^{n_3}$
are n.a. germs, then the composition
$G\circ F: h^{n_1}\rightarrow  h^{n_3}$
also is n.a.
\end{lemma}
\begin{proof}
Denote $F(u)=v$ and $G(v)=w$. Then
$$w_j=G^j(v)=\sum A(G)^j_{\alpha\beta}v^\a \bar v^\b,\qquad
v_l=\sum A(F)^l_{\alpha\beta}u^\a \bar u^\b.$$
Substituting series in series, collecting similar terms and replacing
$u_j$ and $\bar u_j$ by $|u_j|$ we get
$\underline{G\circ F}(|u|)$.
\par Next consider ${\underline G}\circ {\underline F}(|u|)$.
This series is obtained by the same procedure as
$\underline{G\circ F}(|u|)$, but
instead of calculating the modulus of an algebraical sum
of similar terms we take the sum of their moduli.
As $|a+b|\leq |a|+|b|$, we get
$\underline{G\circ F} \leq {\underline G}\circ {\underline F}$.
Since both series
have non-negative coefficients and ${\underline G}\circ {\underline F}$
defines an analytic germ
$ h^{n_1}_R\rightarrow  h^{n_3}_R$,
the assertion follows.
\end{proof}

\begin{lemma}\label{l1.2}
If $F: h^{m}\rightarrow  h^{m}$ is a n.a. germ such that
$F_1= dF(0)=\,$id,
then the germ $G=F^{-1}$ exists and is n.a.
\end{lemma}
\begin{proof}
Write
$\ F(u)=u+F_2(u)+F_3(u)+\dots$.
We are looking for $G(v)$ in the form
$G(v)=v-G_2(v)-G_3(v)-\dots\,.$
Then
\begin{equation*}\begin{split}
F(G(v))&=v-G_2(v)-G_3(v)-\dots \\
&+F_2(v-G_2(v)-\dots, v-G_2(v)-\dots)\\
&+
F_3(v-G_2(v)-\dots,v-G_2(v)-\dots, v-G_2(v)-\dots )+\dots.
\end{split}
\end{equation*}
Here and below we freely identify  $n$-homogeneous maps
with the corresponding $n$-linear symmetric forms.
 Since $F(G(v))=v$, we have the recursive relations
\begin{align*}
G_2(v)&=F_2(v,v),\\
G_3(v)&=F_3(v, v,v)-2F_2(v, G_2(v)),\\
G_4(v)&=F_4(v,v,v,v)-3F_3(v,v,G_2(v))+F_2(G_2(v), G_2(v))
-2F_2(v, G_3(v)),\\
&\dots.\end{align*}
For the same reasons as in the proof of Lemma~\ref{l1.1}
we have
\begin{align*}
\underline G_2(|v|)&\leq \underline F_2(|v|,|v|)=:\check G_2(|v|),\\
\underline G_3(|v|)&\leq
\underline  F_3(|v|, |v|,|v|)+2\underline F_2(|v|, \check G_2(|v|))
=:\check G_3(|v|),\\
&\dots
\end{align*}
These recursive formulas define a germ of an analytic map
 $ h^{m}_R\rightarrow  h^{m}_R$,
$|v|\mapsto \check G(|v|)=|v|+\check G_2(|v|)+\dots$.
Since $\underline G \leq \check G$,  then the assertion follows.
\end{proof}
For a n.a. germ $ F: h^m\rightarrow h^n$
consider its differential, which we regard as a germ
\begin{equation}\label{1.3}
dF(u)v: h^m\times h^m\rightarrow h^n.
\end{equation}

\begin{lemma}\label{l1.3}
Germ \eqref{1.3} is n.a. and
$\ \underline{dF(|u|)|v|}\leq d\underline F(|u|)|v|$.
\end{lemma}
\begin{proof}
Let us write $F$ as series \eqref{series}.
For any $u, v$ we have
$$dF^j(u)(v)=\sum_{\alpha,\beta}\frac{\partial}{\partial t}\bigg |_{t=0}
A^j_{\a\b}(u+tv)^\a(\bar u +t\bar v)^\b$$
$$=\sum_{\alpha,\beta}\sum_r
A^j_{\a\b}
\left(\a_rv_ru^{\a-1_r}\bar u^\b+\b_r\bar v_ru^{\a}\bar u^{\b-1_r}
\right),$$
where $1_r=(0,\dots,0,1,0,\dots)$ (1 is on the $r$-th place).
 Therefore
$$
\underline{dF^j(|u|)(|v|)}\leq \sum_{\alpha,\beta} \sum_r
|A^j_{\a\b}||u|^{\a+\b-1_r}|v_r|(\a_r+\b_r)=d{\underline F}^j(|u|)|v|.
$$
\end{proof}

\begin{lemma}\label{l1.4} i)
The class $\A_{m,\kappa}$ is closed with respect to composition of
germs.

ii)
If $F\in \A_{m,\kappa}$, then
$(id +F)^{-1}=id +G$, where $G\in \A_{m,\kappa}$.

iii) If $F\in\frak A_{m,\kappa}$, then the map $u\mapsto dF(u)u$ also belongs to 
$\frak A_{m,\kappa}$.
\end{lemma}
\begin{proof}
\par  i) If $F, G \in \A_{m,\kappa}$, then
$F\circ G: h^m\rightarrow h^{m+\kappa}$ is n. a. by
Lemma~\ref{l1.1}.
It remains to verify that it satisfies \eqref{1.5}.
We have $d(F\circ G(u))^*=dG(u)^*dF(G(u))^*$.
Arguing as when proving Lemma \ref{l1.1} we get
$$
\underline{d(F\circ G(|u|))^*|v|}
\leq \Phi_G(|G(u)|)\Phi_F(|u|)|v|
\leq \Phi_G(\underline G(|u|))\Phi_F(|u|)|v|.
$$
So $F\circ G$ meets \eqref{1.5}.
\par  ii) Relations, obtained in the proof of
Lemma~\ref{l1.2}, imply that
$G: h^m\rightarrow h^{m+\kappa}$ is n. a. We have
$\ E+dG(u)^*=(E+dF(G(u))^*)^{-1}.$
Therefore
$dG(u)^*=\sum(-1)^k(dF(G(u))^*)^k$
and
$\underline{dG(|u|)^*|v|}
\leq \left(\sum\Phi_G(\underline G(|u|)^k\right)|v|.$
So $dG(u)^*$ satisfies \eqref{1.5} and $G\in \A_{m,\kappa}$.

iii) We skip an easy proof (cf. arguments in the proof of Lemma~\ref{l1.6}). 
\end{proof}

  Let $t\in [0,1]$ and
$V^t(u): [0,1]\times {\cal O}_\delta(h^m)\rightarrow h^{m+\kappa}$
be a continuous map, analytic in $u\in h^m$ and
such that $V^t\in \A_{m,\kappa} \,\,\forall t$, uniformly
in $t$. Consider the equation
\begin{equation}
\label{1.6}
\dot u(t)=V^t(u(t)),\quad u(0)=v,
\end{equation}
and denote by $\varphi^t$, $0\leq t\leq 1$,
its flow maps. That is, $\varphi^t(v)=u(t)$.
\begin{lemma}
\label{l1.5}
For each $0\leq t\leq 1$ we have
$\varphi^t -id \in \A_{m,\kappa}$.
\end{lemma}

\begin{proof}
Denote a solution for \eqref{1.6} as $u=u(t;v)$,
and decompose $u(t;v)$ in series in $v$:
$\ u(t;v)=u_1(t;v)+u_2(t,v)+\dots,$
where $u_k(t;v)$ is $k$--homogeneous in $v$.
Then $u_1(t,v)\equiv v$.
Writing $V^t(u)=V_2^t(u)+V_3^t(u)+\dots$,
we have
$$\dot u_2(t)=V_2^t(u_1,u_1)=V_2^t(v,v),\quad u_2(0)=0.$$
Therefore
$\ u_2(t)=\int_0^tV_2^s(v,v)ds.$
Similar for $k\geq 2$ we have
$$u_k(t)=\sum\limits_{r=2}^{k-1}\sum\limits_{k_1+\dots +k_r=k}
\int\limits_0^{t}V_r^s(u_{k_1}(s),\dots,u_{k_r}(s))ds.$$
Arguing by induction we see that the sum
$\sum_{k=1}^\infty u_k(t,v)$ defines  a n.a.
germ. This  is the germ of the map  $\varphi^t(v)$.
\par For any vector $\xi$,
$d\varphi^t(v)\xi=w(t)$ is a solution of the linearized  equation
$$\dot w(t)=dV^t(u(t))w(t),\quad w(0)=\xi.$$
So $d\varphi^t(v)\xi=U(t)\xi$, where the linear operator
$U(t)$ may be calculated as follows
$$
U(t_0)=\id +\sum\limits_{n=1}^\infty\int\limits_0^{t_0}\int\limits_0^{t_1}
\dots\int\limits_0^{t_{n-1}}dV^{t_1}(u(t_1))\dots dV^{t_n}(u(t_n))
dt_n\dots dt_1.
$$
This series converges if $\|u(0)=v\|_m\ll 1$.
Taking the adjoint to the integral above we see that
$d\varphi^t(u)-id$ satisfies
 \eqref{1.5} and the corresponding operator $\Phi^t(|v|)$
meets the estimate
$$ \Phi^{t_0}(|v|)|\xi|\leq
\sum\limits_{n=1}^\infty\int\limits_0^{t_0}
\dots\int\limits_0^{t_{n-1}} \Phi_{V^{t_n}}(|u(t_n)|)\dots
\Phi_{V^{t_1}}(|u(t_1)|)|\xi|\,dt_n\dots dt_1.$$
Replacing $|u(t_n)|$ by $\underline\varphi^{t_n}(|v|)$
we see that the operator
$ \Phi^t$  defines an analytic germ
$h^m_R\rightarrow {\cal L}(h_R^m,h_R^{m+\kappa})$. So
$\varphi^t -id \in \A_{m,\kappa}$.
\end{proof}

 Let $G_0, F_0\in \A_{m,\kappa}$. Denote
$F(u)=u+F_0(u)$. The arguments in Section~\ref{s4} use the map
$ B(u)=dG_0(u)^*(iF(u))$.

\begin{lemma}\label{l1.6}
$B\in \A_{m,\kappa}$.
\end{lemma}

\begin{proof}
We have $\underline B(|u|)\leq \Phi_{G_0}(|u|)|iF(u)|\leq
\Phi_{G_0}(|u|)\underline F(|u|).$
Since the map in r.h.s. defines an analytic germ
$h^m_R\rightarrow h_R^{m+\kappa}$,
then $B$ is n.a.
\par It remains to check that
$B$ meets \eqref{1.5}. We have $dB(u)\xi=M_1\xi+M_2\xi$,
where
$M_1=dG_0(u)^*idF(u)$
and
$$
M_2=dR(u),\quad
R(u)=
dG_0(u)^*U,\;\;U=iF(u).
$$
Since $M_1^*v=-dF(u)^*idG_0(u)v$, then by Lemma~\ref{l1.3}
$$\underline{M_1(|u|)^*|v|}\leq \left(\Phi_{F_0}(|u|)+E\right)
d\underline{G_0}(|u|)|v|.$$
So $M_1^*v$ has the required form.
 Now consider $M_2$.
 Let $u(t)$ be a smooth curve in $h^m$ such that $u(0)=u$
and $\dot u(0)=\xi$. Then
$$\left<M_2\xi,v\right>=\frac{\partial}{\partial t}\bigg|_{t=0}
\left<dG_0(u(t))^*U,v\right>=$$
$$=\frac{\partial}{\partial t}\bigg|_{t=0}
\left<U, dG_0(u(t))v\right>=\left<U, d^2G_0(u)(\xi,v)\right>.$$
Hence,
$\ M_2^*v=M_2 v=dR(u)v$.
Due to \eqref{1.5} the map $R$ is n.a. and
$\ \underline R(|u|)\le\Phi_{G_0}(|u|)|U|.$
Now Lemma~\ref{l1.3} implies that
$$ \underline{dR(|u|)|v|}
\leq \big(d_{|u|}\Phi_{G_0}(|u|)|v|\big)|U|\leq  \big(d_{|u|}\Phi_{G_0}(|u|)|v|\big)
\big(\underline F_0(|u|)+|u|\big).$$
This component of $ \underline{dB(|u|)^*|v|}$ also has the
required form. So $B$
satisfies \eqref{1.5}.
\end{proof}

\section{Proof of the main theorem}\label{s4}

In this section we prove Theorem~\ref{t2.1}, following the scheme, suggested
in Section~VI of \cite{E90}. To overcome corresponding infinite-dimensional
difficulties we check recursively that all involved germs $\Psi$ of
transformations of the phase-space $h^m$ are of the form $\id+\Psi_0$, where
$\Psi_0\in\frak A_{m,\kappa}$.

 By Lemma~\ref{l1.4} the germ $G=\Psi^{-1}$ is n.a.
and $G=id +G_0$, where $ G_0\in \A_{m,\kappa}$. Denote
$$
\om_1=G^*\om_0,\quad 
\om_\triangle=\om_1-\om_0.
$$
We have
$\om_1=\bar J_1(v) dv\wedge dv$ (see \eqref{2.0}), where
\begin{equation*}
\bar J_1(v)=
i +dG_0(v)^*idG(v) +idG_0(v)=:i+\bar\Upsilon_0(v).
\end{equation*}
Therefore $\omega_1=d\a_1$, where
$$
\a_1(v)\xi=\langle\int_0^1\bar J_1(tv)tv,\xi\rangle\,dt=
\a_0(v)\xi+\langle W(v),\xi\rangle,\quad
W(v)=\int_0^1\bar\Upsilon (tv)tv\,dt
$$
(cf. Lemma 1.3 in \cite{K2} and the corresponding references). So 
$$
\omega_\Delta=d\a_\Delta, \qquad
\a_\Delta=W(v)dv.
 $$
 Lemmas~\ref{l1.6} and \ref{l1.4}~iii) imply that $W\in\frak A_{m,\kappa}$.

\par Our goal  is to find a transformation $\Theta: h^m\rightarrow h^m$
which satisfies i), commutes with the rotations
$u_j\rightarrow e^{i\tau}u_j \quad (j\geq 1,\,\, \tau\in\R)$,
and which ``kills'' the form $\a_\tri$,
thus reducing $\a_1$ to $\a_0$ and $\om_1$
to $\om_0$. Then the mapping $\Psi^+=\Theta\circ \Psi$ would satisfy the required
properties.
We will construct such $\Theta$ in two steps.\medskip

\noindent
{\bf Step 1.}\
At this step we will achieve that the average in angles
of the form $\om_1$ equal to $\om_0$.
\par For $j\geq 1$ and $\tau\in S^1=\R/2\pi \Z$
we define $\Phi_j^\tau: h^m\rightarrow h^m$
as the linear transformation of vectors
$(u_1,u_2,\dots)$ which does not change components $u_l$,
$l\neq j$, and multiplies $u_j$ by $e^{i\tau}$.
Clearly, $(\Phi_j^\tau)^*=\Phi_j^{-\tau}$.
Therefore for a 1-form $\a=F(u)du$
we have
$$(\Phi_j^\tau)^*\a(u)=\left(\Phi_j^{-\tau}F(\Phi_j^\tau(u))\right)du.$$
\par For any function $f(u)$ we define its averaging with respect to
$j$-th angle as
$$M_jf(u)=\frac{1}{2\pi}\int\limits_{0}^{2\pi} f(\Phi_j^t u)dt,$$
and define its averaging in all angles as
$$Mf(u)=(M_1 M_2 \dots)f(u)=\int\limits_{\T^\infty}
f(\Phi^\theta u)d\theta,$$
where $d\theta$ is the Haar measure on $\T^\infty$
and $\Phi^\theta u=(\Phi_1^{\theta_1}\circ \Phi_2^{\theta_2}\circ \dots)u$.
 For a form $\a$ we define $M_j\a$ ad $M\a$ similarly.
That is
$$M_j\a(u)=\frac{1}{2\pi}\int_0^{2\pi}((\Phi_j^t)^*\a)( u)dt,$$
and $M\a=(M_1M_2\dots)\a$. In particular,
$$
M_j(F(u)du)=
\left(\frac{1}{2\pi}\int_0^{2\pi}
\Phi_j^{-\tau}F(\Phi_j^\tau u)\,d\tau\right)du.
$$
\par Since
$$\Phi_j^{\tau^*}\om_1=
\left(\Phi_j^{-\tau}\bar J_1(\Phi_j^\tau v)\Phi_j^\tau\right)dv\wedge dv,$$
then
$$(M\omega_1)(v)= (M{\bar J}_1)(v)dv\wedge dv,\quad
(M\bar J_1)(v)=\int\limits_{\T^\infty}\Phi^{-\theta}
\bar J_1(\Phi^\theta v)\Phi^\theta d\theta.$$

Let us define
$$
(M\bar J)^\tau(v)=(1-\tau)i+\tau  (M\bar J_1)(v).
$$
 The operator $\bar J_1(v)$  is $i+O(v)$. But the averaging in $\theta$
cancels linear in $v$ terms, so
$$
(M\bar J)^\tau(v)=i+\tau \bar \Upsilon(v),\quad \bar\Upsilon(v)=O(v^2).
$$
The operator $\bar\Upsilon(v)\in {\cal L}(h^m,h^{m+\kappa})$
is analytic in $v\in h^m$ and is antisymmetric,
$\bar\Upsilon(v)^*=-\bar\Upsilon(v)$.
So the germ $v\rightarrow \bar\Upsilon(v)\xi$
belongs to
$\A_{m,\kappa}$
for any $\xi\in h^m$.
Cf. the proof of Lemma \ref{l1.5}.

Next we set
$$
 \hat{ J}^\tau(v)=- \big((M\bar J)^\tau(v)\big)^{-1}=-(i+\tau \bar \Upsilon(v))^{-1}.
 $$
 Writing $(i+ \tau \bar\Upsilon(v))^{-1}$
as a Neumann series we see that
$\hat J^\tau(v)=i+\hat\Upsilon^\tau(v),$
where the operator-valued map
$v\mapsto\hat\Upsilon^\tau(v)$ enjoys  the same smoothness properties as
$\bar\Upsilon(v)$.
\par Now consider the average of the  1-form
$\a_\tri=W(v)dv$. We have
$$M\a_\tri=(MW)(v)dv,\quad
(MW)(v)=
\int\limits_{\T^\infty}\Phi^{-\theta}
W(\Phi^\theta v)d\theta.$$
Since $W \in \A_{m,\kappa}$, then also $(MW)\in \A_{m,\kappa}$.
Let us define the mappings
$$V^\tau(v)=\hat J^\tau(v)(MW)(v), \quad 0\leq \tau\leq 1.$$
Due to the properties of $\hat J^\tau(v)$ and $(MW)(v)$,
$V^\tau(v) \in \A_{m,\kappa}$
for each $\tau$. Consider the equation
$$\dot v(\tau)=V^\tau(v(\tau))$$
and denote by $\varphi^\tau$, $0\leq \tau\leq 1$, its flow
maps, $\varphi^\tau(v(0))=v(\tau)$. By Lemma~\ref{l1.5},\,
$\varphi^\tau-id \in \A_{m,\kappa}$.
The operator
$\hat J^\tau(v)$ commutes with the rotations
$\Phi_j^\theta$:
$$\hat J^\tau (\Phi^{\theta_0}v)\Phi^{\theta_0}\xi=
\Phi^{\theta_0}
\hat J^\tau(v)
\xi.$$
The map $(MW)(v)$ also commutes with them.
Accordingly, the maps $V^\tau(v)$ commute with
$\Phi_j^\theta$, as well as the flow maps
$\varphi^\tau$.

 Let us denote
$\hat\om^\tau=(M\bar J)^\tau(v)dv\wedge dv$.
So $\hat\om^1=M\om_1$ and $\hat\om^0=\om_0$.
We claim that
$$(\varphi^\tau)^*\hat \om^\tau=const.$$
To prove this we first note that
$$\frac{d}{d\tau}\hat\om^\tau=M\om_1-\om_0=M(\om_1-\om_0)=
Md\a_\tri=dM\a_\tri.$$
Using the Cartan formula (e.g., see Lemma~1.2 in \cite{K2})
 we have
$$
\frac{d}{d\tau}\big((\varphi^\tau)^*\hat\om^\tau\big)
=(\varphi^\tau)^*\bigg(\frac{\partial\hat\om^\tau}{\partial\tau}+
d(V^\tau\rfloor
\hat\om^\tau)\bigg)=(\varphi^\tau)^* d\big(M\a_\tri+
V^\tau\rfloor\,\hat\om^\tau\big).
$$
The 1-form in the r.h.s. equals
$$(MW)dv+((M\bar J)^\tau V^\tau)dv=(MW)dv-(MW)dv=0,$$
and the assertion follows.
\par Since the maps $\varphi^\tau$ commute
with the rotations $\Phi^\theta_j$, we have
$$
\om_0=
(\varphi^0)^*\hat\om^0=(\varphi^1)^*M\om_1=
M(\varphi^1)^*\om_1.
$$
 Denote $\bar{\Psi}=(\varphi^1)^{-1}\circ\Psi$.
Then $\bar{\Psi}$ satisfies assumptions i), ii)
and in addition
$M\big((\bar\Psi^*)^{-1}\omega_0\big)=\omega_0$.
Since $\varphi^1$ commutes with the rotations,
then $\bar{\Psi}$ satisfies assertion a) of Theorem \ref{t2.1}.
\smallskip
\par We re-denote back $\bar{\Psi}=\Psi$. Then
\smallskip

\noindent
iii) $M\om_1=\om_0$ for $\om_1=\big(\bar{\Psi}^*)^{-1}\om_0$.
\medskip

\noindent
{\bf Step 2.} \
Now we prove the  theorem, assuming that $\Psi$ meets i) -- iii).
Due to iii) we have
$dM\a_\tri=d(M\a_1-\a_0)=M\om_1-\om_0=0$.
Therefore $M\a_\tri=dg$ for a suitable function $g$.
Since $dg=Mdg=dMg$,
we may assume that $g=Mg$. Accordingly,
$\frac{\partial}{\partial\tau}g\big(\Phi_j^\tau(v)\big)=0 \quad \forall\,j.$
Rotations $\Phi_j^\tau$, $\tau\in \R$, correspond to the vector fields
$\chi_j(v)=(0,\dots, iv_j, 0,\dots)$.
So $(dg, \chi_j)=0$
and for each $j$ we have
\begin{equation}\label{2.3}
M(\a_\tri,\chi_j)=(M\a_\tri,\chi_j)=(dg,\chi_j)=0.
\end{equation}

Denote  $h_j(v)=(\a_\tri,\chi_j)$ and consider the system of differential  equations
for a germ of a functional $f:h^m\rightarrow \R$:
\begin{equation}\label{2.4}
(df,\chi_j)\equiv(iv_j \cdot\nabla_{v_j})f(v)=h_j(v),
\quad j\geq 1.
\end{equation}
 First we will check that the vector of the r.h.s.'
$(h_1,h_2,\dots)$ satisfies certain compatibility conditions.
Since
$d\a_0(\chi_i,\chi_j)=\om_0(\chi_i,\chi_j)=0$, then
$$0=d\a_0(\chi_i,\chi_j)=\chi_i(\a_0,\chi_j)-\chi_j(\a_0,\chi_i)-
(\a_0,[\chi_i,\chi_j]),$$
where
$[\cdot,\cdot]$ is the commutator of  vector--fields. Hence,
\begin{equation}\label{2.5}
\chi_i(\a_0,\chi_j)=\chi_j(\a_0,\chi_i).
\end{equation}

\begin{lemma}\label{l2.1}
For any $i$ and $j$ we have
$\chi_i(\a_1,\chi_j)=\chi_j(\a_1,\chi_i)$.
\end{lemma}
\begin{proof}
Recall that $\om_1=\bar J_1(v)dv\wedge dv$,
where
$\bar J_1(v)=i+\bar\Upsilon_0(v)$
and $\bar\Upsilon_0(v)=O(v)$. The operator
$J_1(v)=-\bar J_1(v)^{-1}$
exists for small $v\in h^m$, is antisymmetric and
can be written as $J_1(v)=i+\Upsilon_0(v)$,
where $\Upsilon_0(v)$
belongs to ${\cal L}(h^m,h^{m+\kappa})$.
By interpolation,
$\Upsilon_0(v)\in {\cal L}(h^0,h^{\kappa})$.

To prove the lemma it suffices to show that
\begin{equation}\label{2.100}
\om_1(\chi_i,\chi_j)=0\quad \forall \,i,\,j
\end{equation}
since then the assertion would follow
by the  arguments,  used to establish  \eqref{2.5}.
Moreover, by continuity it suffices to verify the relation at a point
$v=(v_1, v_2,\dots)$ such that $v_j\neq 0$ for all $j$.

 Due to ii), $\{I^j(v),I^k(v)\}_{\om_1}=0$
for any $j$ and $k$.
That is
\begin{equation}\label{2.6}
0=\left<J_1(v)\nabla I^j(v),\nabla I^k(v)\right>=
\left<J_1(v)v_j1_j,v_k1_k\right>
\quad \forall \,j, k.
\end{equation}
Consider the space
$\Sigma_v=\linspan\{v_j1_j,j\ge1\}$
(as before $1_r=(0,\dots,1,\dots)$, where 1 is
on the $r$-th place).
 Its orthogonal complement in $h^0$
is $i\Sigma_v=\linspan\{iv_j1_j, j\ge1\}$.
Relations \eqref{2.6} imply that
$\left<J_1(v)\xi, \eta\right>=0$
for any $\xi,\,\eta\in \Sigma_v$.
Hence,
\begin{equation*}
J_1: \Sigma_v\rightarrow i\Sigma_v.
\end{equation*}
Since $J_1-i=O(v)$, then for small $v$ this linear operator 
is an isomorphism.
As $\chi_i,\,\chi_j\in i\Sigma_v$,
then there exist $\xi_i,\,\xi_j \in \Sigma_v$
such that $J_1\xi_i=\chi_i$, $J_1\xi_j=\chi_j$. So
$\ \om_1(\chi_i, \chi_j)=\left<\bar J_1 J_1\xi_i,  J_1\xi_j\right>=
-\left<\xi_i,  J_1\xi_j\right>=0$
and the lemma is proved.
\end{proof}

 By \eqref{2.5} and the lemma above, relation \eqref{2.5}
also holds with $\a_0$ replaced by $\a_\tri$. That is,
\begin{equation}\label{2.7}
\chi_j(h_k)=\chi_k(h_j) \quad \forall\, j,\,k.
\end{equation}
Also note  that by \eqref{2.3}
\begin{equation}\label{2.8}
Mh_j=0\quad \forall\,j.
\end{equation}
 For any function $g(v)$ and for $j=1,2,\dots$ denote
$$L_jg(v)=\frac{1}{2\pi}\int\limits_0^{2\pi}t\,g(\Phi_j^t(v))dt.$$
 Due to \eqref{2.7},  \eqref{2.8} the system of equations
 \eqref{2.4} is solvable and its solution is given by an explicit
formula due to J.~Moser (see in  \cite{E90}):

\begin{lemma}\label{l2.2}
Consider the germ $f$ of a function in $h^m$:
$$f(v)=\sum\limits_{l=1}^\infty f_l(v),
\quad f_l=M_1\dots M_{l-1}L_lh_l.$$
If the series converges in $C^0(h^m)$,
as well as the series for $\chi_j(f)$, $j\geq 1$,
then $f$ is a solution of \eqref{2.4}.
\end{lemma}
\begin{proof}
For $v=(v_1,v_2,\dots)\in h^m$ and
 $j=1,2,\dots$ let us denote by $\varphi_j$ the argument of $v_j\in\R^2$.
Then
$\chi_j=\frac{\p}{\p\varphi_j}$.
Clearly,
$\frac{\p}{\p\varphi_j}M_jh_j=0$. By \eqref{2.7},
for $k\neq j$ we have
$$\frac{\p}{\p\varphi_k}M_jh_j=M_j\frac{\p}{\p\varphi_k}h_j
=M_j\frac{\p}{\p\varphi_j}h_k=0.$$
So $M_jh_j$ is angle-independent and $M_jh_j=Mh_j=0$
by  \eqref{2.8}

 For any $C^1$-function
$g$ we have
$\frac{\p}{\p\varphi_j}L_jg= L_j\frac{\p}{\p\varphi_j}g=g-M_jg$.
Therefore
$$\frac{\p}{\p\varphi_j}f_k=
\begin{cases}
0,&\quad j<k,\\
M_1\dots M_{k-1}h_k,&\quad j=k,\\
M_1\dots M_{k-1}h_j-M_1\dots M_{k}h_j,&\quad j>k,
\end{cases}
$$
(for $k=1$ we define $M_1\dots M_{k-1}h_j=h_j$).
So $\frac{\p}{\p\varphi_j}\sum f_k=h_j$.
\end{proof}
\par Since $\a_\tri=W(v)dv$, then
$$
h_j=(\alpha_\Delta,\chi_j)=iv_j\cdot W_j(v).
$$
The estimates on $W(v)$ easily imply
that the series for $f$ and $\chi_kf$,
$k\geq 1$, converge.
So $f$ is a solution of \eqref{2.4}.
Let us consider its differential
$\ df=\nabla_vf(v)dv.$
Here $\nabla_vf(v)=(\xi_1,\xi_2,\dots)$,
where
$\xi_j=\frac{\p f}{\p v_j^+}+i\frac{\p f}{\p v_j^-}$
with $v_j=v_j^++iv_j^-$.

\begin{lemma}\label{l2.3}
The germ
$v\mapsto Y(v)=\nabla_vf(v)$, $h^m\rightarrow h^{m+\kappa}$,
is n.a. and $Y(v)=O(v)$.
\end{lemma}
\begin{proof}
Noting that $iv_j\cdot W_j=\Phi_j^{\pi/2}v_j\cdot W_j$, we
 have $\nabla_{v_i}f=\sum_j\nabla_{v_i}f_j$,
where
$$\nabla_{v_i}f_j=\int\limits_{\T^j}\theta_j\nabla_{v_i}
\bigg(W_j(\Pi^{\theta^j}v)\cdot
(\Phi_j^{\theta_j+\frac{\pi}{2}}v_j)\bigg)d\theta^j$$
$$=\int\limits_{\T^j}\theta_j(\Phi_i^{-\theta_i}\nabla_{v_i})
W_j(\Pi^{\theta^j}v)\cdot(\Phi_j^{\theta_j+\frac{\pi}{2}}v_j)d\theta^j+
\delta_{ij}\int\limits_{\T^j}\theta_j
\big(\Phi_j^{-\theta_j-\frac{\pi}{2}}W_j(\Pi^{\theta^j}v)\big)=:
$$
$$=:Y_{ij}+\delta_{ij}Z_j.
$$
Here
$$
\theta^j=(\theta_1,\dots,\theta_j)\in \T^j,\quad
\Pi^{\theta^j}=\Phi_1^{\theta_1}\circ\dots \circ\Phi_j^{\theta_j},\quad
d\theta^j=\frac{d\theta_1}{2\pi}\dots \frac{d\theta_j}{2\pi}\,.
$$
 Denote by ${\cal Z}_j^\theta(v)$ the integrand for $Z_j$.
The germ $v\mapsto {\cal Z}^\theta(v)$, $ h^m\rightarrow h^{m+\kappa}$,
is analytic and
$\underline{{\cal Z}^\theta}(|v|)=(\diag\,\theta_j)\underline W(|v|)$.
Hence, the germ ${\cal Z}^\theta$ is n.a.
for each $\theta$. So the germ
$v\mapsto Z(v)$ also is.
\par Denote by ${\cal Y}^\theta_{ij}(v)$ the integrand in $Y_{ij}$.
We have
$${\cal Y}^\theta_{ij}(v)=\theta_j
\Phi_i^{-\theta_i}\big(dW(\Pi^{\theta^j}(v))^*\big)_{ij}
\Phi_j^{\theta_j+\frac{\pi}{2}}v_j.$$
Using \eqref{1.5} we see that
$\big|\sum\limits_j{\cal Y}^\theta_{ij}(|v|)\big|\leq2\pi\,
\sum\limits_j(\Phi_W)_{ij}(|v|)|v_j|,$
uniformly in $\theta$.
That is, for any $\theta$ the germ
$h^m\ni v\mapsto \sum\limits_j{\cal Y}^\theta_{ij}(v)\in h^{m+\kappa}$
is n.a.
So the map $h^m\ni u\mapsto Y(u)=(Y^1, Y^2,\dots)\in h^{m+\kappa}$
 is n.a. It is obvious that $Y(u)=O(u)$.
\end{proof}
\par For $0\leq\tau\leq 1$ we set
$ \bar J^\tau(v)=(1-\tau)i+\tau\bar J_1(v)$
and
$J^\tau(v)=-\big(\bar J^\tau(v))^{-1}$.
These are well defined operators in $ {\cal L}(h^m,h^{m})$
antisymmetric with respect to the $h^0$-scalar product.
Clearly, $J^\tau(v)=i+\Upsilon^\tau(v)$, where
$\Upsilon^\tau(v)$ belongs to ${\cal L}(h^m,h^{m+\kappa})$.
Denote by $\om^\tau$ the form
$(1-\tau)\om_0+\tau\om_1=\bar J^\tau(v)dv\wedge dv$.

 We define
$V^\tau(v)=J^\tau(v)(W(v)-Y(v))$,
consider the equation
\begin{equation}\label{flow}
\dot v(\tau)=V^\tau(v(\tau)),\quad 0\leq\tau\leq1,
\end{equation}
and denote by $\varphi^\tau$, $0\leq\tau\leq1$, its flow-maps.
Since $V^\tau(v)=O(v)$ and $V^\tau:h^m\mapsto h^{m+\kappa}$ is n.a.
(cf. Lemma~\ref{l1.5} and its proof), then $(\varphi^\tau-id)=O(v)$ and
$(\varphi^\tau-id):h^m\mapsto h^{m+\kappa}$ is n.a.
 Also
$$
\frac{d}{d\tau}(\varphi^\tau)^*\omega^\tau=
(\varphi^\tau)^*d(\alpha_\Delta+V^\tau\rfloor\omega^\tau)=
(\varphi^\tau)^*d\big(Y(v)dv\big)=0
$$
since $Y(v)dv=df$. So  $(\varphi^\tau)^*\om^\tau=const$ and
 $\ (\varphi^0)^*\om^0=\om_0=(\varphi^1)^*\om^1.$
That is, the n.a. germ
$$
\Psi^+=(\varphi^1)^{-1}\circ\Psi: h^m\rightarrow h^m
$$
is such that $d\Psi^+(0)=id$  and $\Psi^{+^*}\om_0=\om_0$.
The germ  $(\varphi^1)^{-1}-\,$id$\,:h^m\mapsto h^{m+\kappa}$ is
n.a., so $\Psi^+_0=\Psi^+-\,$id$\,:h^m\mapsto h^{m+\kappa}$ is n.a.
as well.

Now we show that
$\Psi^+_0\in\A_{m,\kappa}$.
Since $\Psi^+$ is symplectic,
then\\
$\ d\Psi^+(u)^*\, i\big(d\Psi^+(u)\big)=i.$
Hence,
$$d\Psi_0^+(u)^*= id\Psi_0^+(u)  (1+d\Psi_0^+(u))^{-1}i,$$
and   Lemma \ref{l1.3} implies that
$\underline{d\Psi_0^+(|u|)^*|v|}$ has the required form \eqref{1.5}.

It remains to check for $\Psi^+$ properties ii) and a).
Let $v$ be a vector such that $v_j\ne0$ for all
$j$. Since $\a_\tri=W(v)dv$, then for each $j$
we have
\begin{equation}\label{2.10}
\begin{aligned}
\om^\tau(V^\tau, iv_j1_j)=&\left<\bar J^\tau(v)V^\tau(v), iv_j1_j\right>=\\
=\left<W(v)-Y(v), iv_j1_j\right>&=(\a_\tri,\chi_j)-
(df,\chi_j)=0.
\end{aligned}
\end{equation}
 By \eqref{2.100} and a similar relation for
the form $\om_0$,
\begin{equation}\label{2.11}
\om^\tau(\xi_1,\xi_2)=0
\quad \forall\,\xi_1,\,\xi_2\in i\Sigma_v.
\end{equation}
Here as before $i\Sigma_v=\linspan\{iv_j1_j\}$.
Denote
$$\big(i\Sigma_v\big)^\perp=\{\xi\in h^0\,\,\big|\, \om^\tau(\xi,\eta)=0
\,\,\forall\eta\in i\Sigma_v\}.$$
By \eqref{2.11},  $i\Sigma_v \subset \big(i\Sigma_v\big)^\perp$.
We claim that
\begin{equation}\label{2.12}
i\Sigma_v = \big(i\Sigma_v\big)^\perp
\end{equation}
(i.e., $i\Sigma_v$ is a Lagrangian subspace for the symplectic form $\omega^\tau$).
Indeed, if this is not the case,
then we can find a vector $\xi\in\Sigma_v$,
$\|\xi\|_0=1$,
such that $\xi\in \big(i\Sigma_v\big)^\perp$.
In particular,
$\om^\tau(\xi,i\xi)=0$.
But $\om^\tau(0)(\xi,i\xi)=\om_0(\xi,i\xi)=1$.
So for small $v$ we have
$\om^\tau(\xi,i\xi)>0$. Contradiction.
\par
By \eqref{2.10}, \eqref{2.12},
$V^\tau(v)\in  \big(i\Sigma_v\big)^\perp=i\Sigma_v$.
So  solutions $v(\tau)$  of \eqref{flow} satisfy
$$
\frac{1}{2}\frac{d}{d\tau}|v_j(\tau)|^2=
\left<V^\tau(v), v_j1_j\right>=0,$$
and $I_j(\varphi^\tau(v))\equiv I_j(v)$ for each $j$. By continuity
 this relation holds
for all vectors $v$ (without assuming that $v_j\ne0\ \forall\,j$). Hence,
$I_j\circ\Psi\equiv I_j\circ\Psi$ for each $j$.
This proves ii) and  a) for the germ $\Psi^+$.

\section{Proof of Theorem \ref{t3.1}} \label{s5}
\par The construction of $\Psi$ that we present below
follows  the ideas of \cite{Kap91} (also see \cite{K2}, pp.~42-44).
It relies on the spectral theory of the corresponding Lax operator
$L_{u}=-\p^2_x-u.$
\par It will be convenient for us to allow for complex-valued potentials $u$:
$$u\in L_2^0(S^1,\C)=\{u\in L_2(S^1,\C),\quad \int_0^{2\pi}u dx=0\}.
$$
We write $u(x)$ as Fourier series
$u(x)=\frac{1}{2\sqrt{\pi}}\sum_{j\in \Z_0}
e^{ijx}w_j$,
where $\Z_0=\Z\smallsetminus\{0\}$, and denote
\begin{equation*}
w=(w_j\in \C,\,\, j\in \Z_0)={\cal F}(u),\quad u={\cal F}^{-1}(w).
\end{equation*}
Clearly,
${\cal F}:  L_2^0(S^1,\C)\rightarrow h^0=h^0(\Z_0)$ is an isomorphism.
Here and below
 we use the notations
$$h^m=h^m(\Z_0)=
\{w: \|w\|_m^2=\sum\limits_{j\in \Z_0}|j|^{2m}|w_j|^2<\infty\}.$$
A sequence $w$ called {\it real} if
${\cal F}^{-1}(w)$ is a real-valued function.
That is, if $w_j=\overline{ w_{-j}}$ for each $j$.

\par We view $L_u$ as an operator on $L_2(\R/4\pi\Z)$ with the domain
${\cal D}(L_u)=H^2(\R/4\pi\Z)$. The spectrum of $L_u$ is discrete and
for $u$ real is of the form
$$\sigma(L_u)=\{\lambda_k(u), \,\,k\geq 0\},
\quad \lambda_0<\la_1\leq\la_2<\dots,$$
where
$\la_k(u)\rightarrow \infty$ as $k\rightarrow \infty$.
For $u$ small $\s(L_u)$ is $\|u\|_{L_2}$- close
to the spectrum of $L_0=-\p^2_x$, that is, to the set
$\s(L_0)=\{j^2/4,\, \, j\geq 0\}$. More precisely, one has
$$|\la_{2j-1}-j^2/4|,\,  |\la_{2j}-j^2/4|\leq C\|u\|_{L_2},
\quad j\geq 1,$$
provided $\|u\|_{L_2}\leq \delta$, where $\delta>0$ sufficiently small.

\par For $j\geq 1$   we will denote by $E_j(u)$ the invariant two-dimensional
subspace of $L_u$, corresponding to the eigenvalues
$\la_{2j-1}(u)$, $\la_{2j}(u)$, and by $P_j(u)$
the spectral projection on $E_j(u)$:
\begin{equation*}
P_j(u)=-\frac{1}{2\pi i}\oint\limits_{\gamma_j}(L_u-\la)^{-1}d\la,
\quad \Im P_j(u)=E_j(u),\quad j\geq 1.
\end{equation*}
Here $\gamma_j$ is a contour in the complex plane which isolates
$\la_{2j-1}$ and $\la_{2j}$ from other eigenvalues of $L_u$.
For the computations that will be performed below we fix
the contours  as
$\gamma_j=\{\la\in\C,\,\, |\la-j^2/4|=\delta_0 j\}$,
$\delta_0>0$ small.

\par Clearly, $u\mapsto P_j(u)$, $j\geq 1$, are analytic\footnote{In this section
``analytic'' means complex analytic.}
maps from
$V_\delta=\{u\in L_2^0(S^1,\C),$ $\|u\|_{L_2}\leq \delta\}$ to
${\cal L}(L_2,H^2)$, $L_2=L_2(\R/4\pi\Z)$, $H^2=H^2(\R/4\pi\Z)$,
provided $\delta$ is sufficiently small.
Furthermore, it is not difficult to check that
\begin{equation}\label{4.2}
\begin{split}
\|P_j(u)-P_{j0}\|_{L_2\rightarrow L_2}   \leq Cj^{-1}\|u\|_{L_2},\quad
\|P_j(u)-P_{j0}\|_{L_2\rightarrow H^2}   \leq Cj\|u\|_{L_2},
\end{split}
\end{equation}
for $j\geq 1$ and $u\in V_\delta$. Here $P_{j0}$ is the spectral projection
of $L_0$,  corresponding to a double eigenvalue $j^2/4$:
$$\Im P_{j0}=E_{j0},\quad \Ker P_{j0}=E_{j0}^\perp,\quad
E_{j0}=\linspan\{\cos{jx/2},\, \sin{jx/2}\},\, j\geq 1.
$$
\par Following \cite{Ka}, see also \cite{Kap91},
we introduce the transformation operators
$U_j(u)$, $j\geq 1$:
\begin{equation*}\label{4.00}
U_j(u)=\big(I-(P_j(u)-P_{j0})^2\big)^{-1/2}P_j(u).
\end{equation*}
It follows from \eqref{4.2} that the maps
$u\mapsto U_j(u)$ are well defined and analytic on $V_\delta$.
It turns out (see \cite{Ka}) that the image of $U_j(u)$ is $E_j(u)$ and
for $u$ real one has
\begin{align}
\|U_j(u)f\|_2&=\|f\|_2,\quad f\in E_{j0},\label{4.3}\\
\overline{U_j(u)f}&=U_j(u)\bar f.\label{4.4}
\end{align}

\par For $j\in \Z_0$ let us set
\begin{align}
f_j(u)&=U_{|j|}(u)f_{j0}\in E_{|j|}(u),\quad
f_{j0}=\frac{1}{\sqrt{2\pi}}e^{-ijx/2},
\label{4.50}
\\
z_j(u)&=-\sqrt{\pi}\left((L_u-j^2/4)f_j(u),\overline{f_j(u)}\right).
\label{4.5}
\end{align}
Here $(\cdot,\cdot)$ stands for the standard scalar product in
$L_2([0,4\pi],\C)$:
$$
(f,g)=\int f\bar g\, dx.
$$
\begin{lemma}\label{l4.1}
For $u$ real, one has
\begin{itemize}
\item[i)] $\overline{z_j(u)}=z_{-j}(u)$,
\item[ii)] $|z_j(u)|^2=\pi(\la_{2j}(u)-\la_{2j-1}(u))^2,\quad j\geq 1.$
\end{itemize}
\end{lemma}
\begin{proof}
Assertion i) is obvious (see \eqref{4.4}). To check ii),
consider
$$
\begin{aligned}
e_j&=\Re f_j=U_j(u)e_{j0},\quad e_{j0}=\frac{1}{\sqrt{2\pi}}\cos{jx/2},\\
e_{-j}&=\Im f_j=U_j(u)e_{-j0},\quad
e_{-j0}=-\frac{1}{\sqrt{2\pi}}\sin{jx/2},\quad j\geq 1.
\end{aligned}
$$
It follows from \eqref{4.3},  \eqref{4.4} that
the vectors $e_j$, $e_{-j}$ form a real orthonormal basis of $E_j(u)$.
Let $M_j(u)$  be the matrix of the self-adjoint operator
$-\sqrt{\pi}(L_u-j^2/4)\big|_{E_j(u)}$ in this basis:
$$M_j(u)=\left(\begin{array}{cr}
a_j^1&b_j\\
b_j&a_j^2
\end{array}
\right),
$$
$$ a_j^1=-\sqrt{\pi}
\left((L_u-\frac{j^2}{4})e_j(u),{e_j(u)}\right),\;\;
a_j^2=-\sqrt{\pi}
\left((L_u-\frac{j^2}{4})e_{-j}(u),{e_{-j}(u)}\right),$$
$$
b_j=-\sqrt{\pi}
\left((L_u-j^2/4)e_j(u),{e_{-j}(u)}\right)=\frac{1}{2}\Im z_j(u).$$
Consider the deviators $M_j^D,\,\, j\geq 1$,
(for a $2\times 2$ matrix M its deviator is the traceless matrix
$M-(\frac{1}{2}\tr M)I$):
$$M_j^D(u)=\left(\begin{array}{cr}
a_j&b_j\\
b_j&-a_j
\end{array}
\right),\quad a_j=\frac{1}{2}(a_j^1-a_j^2)=\frac{1}{2}\Re z_j(u).$$
By construction, one has
$|z_j(u)|^2=4(a_j^2+b_j^2)=\pi(\la_{2j}(u)-\la_{2j-1}(u))^2$.
\end{proof}
Functions $z_j(u)$, $j\in\Z_0$, are analytic functions of $u\in V_\delta$,
vanishing at zero.
They can be represented by absolutely and uniformly
converging Taylor series that we will write in terms of the
Fourier coefficients $w={\cal F}(u)$.
\begin{equation}\label{4.7}
z_j(u)=\sum\limits_{n=1}^\infty Z_n^j(w),
\end{equation}
where $Z_n^j(w)$ are bounded $n$- homogeneous functionals on $h^0(\Z_0)$:
\begin{equation}\label{4.9}
Z_n^j(w)=\sum\limits_{{i}=(i_1,i_2,\dots, i_n)\in \Z_0^n}
{\cal K}_n^j({i})w_{i_1}w_{i_2}\dots w_{i_n},
\end{equation}
${\cal K}_n^j(\cdot)$ being a symmetric function on $\Z_0^n$.

 Notice that
$$P_j(u)=P_{j0}+(L_0-j^2/4)^{-1}(I-P_{j0})uP_{j0}+
P_{j0}u(L_0-j^2/4)^{-1}(I-P_{j0})+O(u^2),$$
$$U_j(u)=P_j(u)+O(u^2).$$
As a consequence,
\begin{equation}\label{4.10}
f_j(u)=f_{j0}+(L_0-j^2/4)^{-1}(I-P_{|j|0})uf_{j0}+O(u^2),\quad j\in \Z_0.
\end{equation}
Substituting \eqref{4.10} into \eqref{4.5}, one gets
$$z_j(u)=\sqrt{\pi}(uf_{j0},\bar f_{j0})+
\sqrt{\pi}(u(L_0-j^2/4)^{-1}(I-P_{|j|0}) uf_{j0},\bar f_{j0})+O(u^3),$$
which gives that $Z_1(w)=w$ and
\begin{equation}\label{4.12}
Z_2^j(w)=\frac{1}{2\sqrt{\pi}}\sum\limits_{k\in\Z_0,\,k\neq j}
\frac{w_kw_{j-k}}{k(k-j)}
\end{equation}
In a similar way, one can show that
\begin{equation}\label{4.100}
Z_3^j(w)=\frac{1}{4{\pi}}
\sum\limits_{{i_1,i_2\in \Z_0\atop i_1\neq j, i_1+i_2\neq 0,j}}
\frac{w_{i_1}w_{i_2}w_{j-i_1-i_2}}{i_1(i_1-j)(i_1+i_2)(i_1+i_2-j)}
-\frac{w_j}{4{\pi}}\sum\limits_{l\in\Z_0,\,l\neq j}
\frac{w_lw_{-l}}{l^2(l-j)^2}.
\end{equation}
\par The structure of higher order terms
$Z_n^j(w)$ is described by the following lemma which is the key
technical step
of our analysis.
\begin{lemma}\label{l4.2}
One has
\begin{itemize}
\item[(i)] $\supp{\cal K}_n^j(\cdot)\subset \Omega_j^{(n)}$,
where $ \Omega_j^{(n)}$ is the simplex
$$\Omega_j^{(n)}=\{ i=(i_1,\dots, i_n)\in \Z_0^n,\,\,
\sum\limits_{l=1}^ni_l=j\};$$
\item[(ii)] for $n\geq 2$, \
$\|{\cal K}_n^j\|_{l^2(\Z_0^n)}\leq {R^n}{|j|^{1-n}};$
\item[(iii)] for $n\geq 3$,
$\|{\cal B}_n^j\|_{l^2(\Z_0^n)}\leq {R^n}{|j|^{-2}},$
where
${\cal B}_n^j(i_1,\dots, i_n)={\cal K}_n^{i_1}(j,i_2,i_3,\dots,i_n).$
\end{itemize}
Here $R$ is a positive constant, independent of $j$ and $n$.
\end{lemma}
Postponing the proof of this lemma till the end of the section,
we proceed with the construction of the map $\Psi$.
Introduce the map $F$ that associates to $w\in h^0(\Z_0)$ the sequence
$F(w)=(z_j(u), \, j\in \Z_0)$, $u={\cal F}^{-1}(w)$.
Since $Z_1(w)=w$,  we write $F$ as the sum
$$F=\id+F_2,\quad F_2=Z_2+F_3,$$
where
$Z_2(w)=(Z_2^j(w),\, j\in \Z_0)$.
Notice that, by the construction,
$$ dF(0)=id, \quad F_2(w)=O(w^2),\quad F_3(w)=O(w^3).$$
\par As a direct consequence of Lemma \ref{l4.2} (i), (ii) one gets
\begin{lemma}\label{l4.3}
For $j=2,3$ the map
$F_j:{\cal O}_\delta(h^m(\Z_0))\rightarrow h^{m+j-1}(\Z_0)$
is analytic and normally analytic.
\end{lemma}
\begin{proof}
It is sufficient to show that for  any $n\geq 2$,
$$\|\underline{Z_n}(v)\|_{m+n-1}\leq R^n\|v\|_m^n,
\quad v\in h_R^m(\Z_0).$$
Here
$\underline{Z_n}(v)=(\underline{Z_n^j}(v),\, j\in \Z_0)$ and
$$\underline{Z_n^j}(v)=\sum\limits_{{i}=(i_1,i_2,\dots, i_n)\in \Z_0^n}
\big|{\cal K}_n^j({i})\big|v_{i_1}v_{i_2}\dots v_{i_n}.$$
From Lemma \ref{l4.2} (i), (ii) and the Cauchy-Schwartz inequality we get
\begin{equation*}\begin{split}
\big|\underline{Z_n^j}(v)\big|^2
&\leq \|{\cal K}_n^j\|_{l^2(\Z_0^n)}^2
\left(\sup\limits_{ i\in \Omega_j^{(n)}}\Lambda^{-2m}( i)\right)
\sum\limits_{ i=(i_1,\dots,i_n)\in\Omega_j^{(n)}}
\Lambda^{2m}( i)v^2_{i_1}v^2_{i_2}\dots v^2_{i_n}\\
&\leq \frac{R^{2n}}{|j|^{2n-2+2m}}
\sum\limits_{ i=(i_1,\dots,i_n)\in\Omega_j^{(n)}}
\Lambda^{2m}(i)v^2_{i_1}v^2_{i_2}\dots v^2_{i_n}.
\end{split}
\end{equation*}
Here $\Lambda( i)=|i_1||i_2|\dots |i_n|$
and at the last step we have used the inequality
\begin{equation}\label{4.13}
\sup\limits_{i\in \Omega_j^{(n)}}\Lambda^{-1}(i)\leq
{R^n}{|j|^{-1}}.
\end{equation}
Summing up with respect to $j$ one gets that
$
\|\underline{Z_n}(v)\|_{m+n-1}\leq R^n\|v\|_m^n$ for  $ n\geq 2$.
\end{proof}
\par Consider $dF_3(w)^t$ -- the real transposed of $dF_3(w)$ with respect to
the (standard complex) scalar product in $h^0(\Z_0)$:
$$(dF_3(w)h, \bar g)_{h^0(\Z_0)}=(h, \overline{dF_3(w)^tg})_{h^0(\Z_0)}.$$
Notice that for analytic maps the inequality of Lemma \ref{l1.3}
becomes an equality:
$$\underline{dF(|w|)|h|}=d\underline{F}(|w|)|h|,\quad
\underline{dF(|w|)^t|h|}=d\underline{F}(|w|)^t|h|.$$
\begin{lemma}\label{l4.4}
\begin{itemize}
\item[(i)] For $w\in {\cal O}_\delta(h^m)$,
$dF_3(w)^t\in {\cal L}(h^m, h^{m+2})$
and the map
$w\mapsto dF_3(w)^t$,
${\cal O}_\delta(h^m)\rightarrow  {\cal L}(h^m, h^{m+2})$,
is analytic;
\item[(ii)] similarly,
$d\underline{F_3}(v)^t$, $v\in {\cal O}_\delta(h_R^m)$, belongs to
${\cal L}(h^m_R, h^{m+2}_R)$
and the map $v\mapsto d\underline{F_3}(v)^t$,
${\cal O}_\delta(h^m_R)\rightarrow  {\cal L}(h^m_R, h^{m+2}_R)$,
is real analytic.
\end{itemize}
\end{lemma}
\begin{proof}
We have
$$
\big(dZ_n(w)^th\big)_j=nB_n^j(h,w,\dots, w),\quad
\big(d\underline{Z_n}(v)^tg\big)_j=n\underline{B}_n^j(g,v,\dots, v)
$$
Here
$B_n^j$ and $\underline{B}_n^j$ are the $n$-linear forms with the kernels
${\cal B}_n^j$ and $\big|{\cal B}_n^j\big |$ respectively:
$$B_n^j(w^1,\dots,w^n)=\sum\limits_{i=(i_1,\dots,i_n)\in \Z_0^n}
{\cal B}_n^j(i)w^1_{i_1}\dots w^n_{i_n},\quad w^k\in h^m \quad\forall\,k,$$
$$\underline{B}_n^j(v^1,\dots, v^n)=
\sum\limits_{i=(i_1,\dots,i_n)\in \Z_0^n}
\big|{\cal B}_n^j(i)\big|v^1_{i_1}\dots v^n_{i_n},
\quad v^k\in h^m_R \quad\forall\,k.$$
To prove the lemma it is sufficient to show that for $n\ge3$
 the poly-linear map
$$\underline{B}_n=
(\underline{B}_n^j,\,j\in \Z_0):
h_R^m\times\dots \times h_R^m\rightarrow h_R^{m+2}$$
is bounded and verifies:
\begin{equation}\label{010}
\|\underline{B}_n(v^1,\dots , v^n)\|_{m+2}\leq
R^n\prod\limits_{k=1}^n\|v^k\|_m,\quad v^k\in h_R^m,\,\, k=1,\dots, n.
\end{equation}
It follows from Lemma \ref{l4.2} (i)
that
\begin{equation*}
\begin{split}
\big |\underline{B}_n^j(v^1,\dots, v^n)\big|^2\leq
\|{\cal B}_n^j\|_{l^2(\Z_0^n)}^2
\big(\sup\limits_{i\in \Omega_{-j}^{(n)}}\Lambda^{-2m}(i)\big)\\
\times
\sum\limits_{{i=(i_1,\dots,i_n)\in \Z_0^n\atop
(-i_1,i_2,i_3,\dots, i_n)\in\Omega_{-j}^{(n)}}}
\Lambda^{2m}(i)\big (v^1_{i_1}v^2_{i_2}\dots v^n_{i_n}\big)^2.
\end{split}
\end{equation*}
Combining this inequality with (ii) of Lemma \ref{l4.2}
and using once more \eqref{4.13} we get \eqref{010} for any $n\ge3$.
\end{proof}

We next denote by $\tilde D$ the operator of
multiplication by the diagonal matrix
$\diag(|j|^{1/2},\, j\in \Z_0)$.
It defines isomorphisms
$\tilde D: h^r\rightarrow h^{r-1/2}$, $r\in \R$. Let us set
$m^\prime=m+\frac{1}{2}\geq\frac{1}{2} $. For any analytic germ
$H: h^m\rightarrow h^{m+a}$ we will denote by $H^{\tilde D}$ the germ
$H^{\tilde D}=
\tilde D^{-1}\circ H\circ \tilde D: h^{m^\prime}\rightarrow  h^{m^\prime+a} $
In particular, $F^{\tilde D}=\tilde D^{-1}\circ F\circ \tilde D$.
Due to Lemma \ref{l4.3}  one has
\begin{itemize}
\item[(a)]
$F^{\tilde D}: {\cal O}_\delta(h^{m^\prime})\rightarrow h^{m^\prime}$ is
n.a.;
\item[(b)]
$F^{\tilde D}-id=F_2^{\tilde D}:{\cal O}_\delta(h^m)\rightarrow h^{m+1}$
is n.a. and  $F_2^{\tilde D}(v)=O(v^2)$.
\end{itemize}
Notice that the operations $F\mapsto F^{\tilde D}$
and $F\mapsto \underline{F}$ commute.
\par Consider $(dF_2^{\tilde D})^t$, $(d\underline{F_2^{\tilde D}})^t$. We have
\begin{lemma}\label{l4.5}
\begin{itemize}
\item[(i)] For $v\in {\cal O}_\delta(h^{m^\prime})$,
$dF_2^{\tilde D}(v)^t\in {\cal L}(h^{m^\prime}, h^{m^\prime+1})$
and the map
$v\mapsto dF_2^{\tilde D}(v)^t$,
${\cal O}_\delta(h^{m^\prime})\rightarrow
{\cal L}(h^{m^\prime}, h^{m^\prime+1})$,
is analytic;
\item[(ii)] similarly,
$d\underline{F_2^{\tilde D}}(v)^t$,
$v\in {\cal O}_\delta(h_R^{m^\prime})$, belongs to
${\cal L}(h^{m^\prime}_R, h^{m^\prime+1}_R)$
and the map $v\mapsto d\underline{F_2^{\tilde D}}(v)^t$,
${\cal O}_\delta(h^{m^\prime}_R)\rightarrow
{\cal L}(h^{m^\prime}_R, h^{m^\prime+1}_R)$,
is real analytic.
\end{itemize}
\end{lemma}
\begin{proof}
As in the proof of Lemma \ref{l4.4},
it is sufficient to prove the statement, corresponding to
$d\underline{F_2^{\tilde D}}(v)^t$.
We write
$$F_2^{\tilde D}=Z_2^{\tilde D}+F_3^{\tilde D},\quad
\underline{F_2^{\tilde D}}=
\underline{Z_2^{\tilde D}}+\underline{F_3^{\tilde D}}.$$

Since $d\underline{F_3^{\tilde D}}(v)^t=
\tilde Dd\underline{F_3}^t(\tilde D v)\tilde D^{-1}$,
Lemma \ref{4.4} implies that
the map $v\mapsto d\underline{F_3^{\tilde D}}(v)^t$,
$ {\cal O}_\delta(h^{m^\prime}_R)\rightarrow
{\cal L}(h^{m^\prime-1}_R,h^{m^\prime+1}_R)$, is real analytic.
Therefore it is also real analytic as a map from
${\cal O}_\delta(h_R^{m^\prime})$
to ${\cal L}(h^{m^\prime}_R,h^{m^\prime+1}_R)$.
\par Next consider
$\underline{Z_2^{\tilde D}}=\tilde D^{-1}\underline{Z_2}\tilde D$.
Note that $\underline{Z_2}(w)=\const\, \tilde D^{-2}w*\tilde D^{-2}w$.
Accordingly,
$d\underline{Z_2}(w)(f)=\const\, \tilde D^{-2}w*\tilde D^{-2}f$,
and
$$d\underline{Z_2^{\tilde D}}(v)^t(f)=\const\tilde D^{-1}
\big(\tilde D^{-1}v*\tilde D^{-1}f\big).$$
Since $\tilde D^{-1}: h^{m^\prime}\rightarrow  h^{m^\prime+1/2}$,
where $m^\prime+1/2\geq 1$, and since
the convolution defines a continuous bilinear
map $h^r\times h^r\rightarrow h^r$ if $r>1/2$,
then we have
$$\|d\underline{Z_2^{\tilde D}}(v)^t(f)\|_{m^\prime+1}
\leq C_m\|v\|_{m^\prime}\|f\|_{m^\prime}.$$
So the map $h^{m^\prime}\ni v\mapsto
d\underline{Z_2^{\tilde D}}(v)^t\in
{\cal L}(h^{m^\prime}_R, h^{m^\prime+1}_R)$
is bounded, which concludes the proof of
Lemma \ref{l4.5}.
\end{proof}
\par We are now in position to finish the proof of Theorem \ref{t3.1}.
Define $\Psi:
{\cal O}_\delta(h^{m^\prime}(\N))\rightarrow h^{m^\prime}(\N)$ by restricting
$F^{\tilde D}$ on
the subspace of real sequences
$\{v=(v_j,\,j\in\Z_0)\in h^{m^\prime}(\Z_0),\,\, v_j=\overline{v_{-j}}\}$
(that is on the real potentials $u$) and further projecting
it on the positive indices:
$$\Psi=\pi F^{\tilde D}\circ \pi^{-1}.$$
Here $\pi:h^{m^\prime}(\Z_0)\rightarrow h^{m^\prime}(\N)$ is the projection:
$$\pi: v=(v_j,\,j\in \Z_0)\mapsto \pi v=(v_j,\,j\geq 1),$$
and $\pi^{-1}:  h^{m^\prime}(\N)\rightarrow h^{m^\prime}(\Z_0)$ is the
right inverse map:
$$\pi^{-1}:v=(v_j,\,j\geq 1)\mapsto v^\prime=(v^\prime_j,\,j\in \Z_0),$$
$$
v^\prime_j=v_j \,\,{\rm for}\,\, j\geq 1, \,\,v^\prime_j=\overline{v_{-j}}\,\,
\rm{for}\,\, j\leq -1.$$
Clearly,
$\Psi:{\cal O}_\delta(h^{m^\prime})\rightarrow h^{m^\prime}$,
$h^{m^\prime}=h^{m^\prime}(\N)$, is a real analytic map of the form
$\Psi=id+\Psi_2$,
where $\Psi_2=\pi F_2^{\tilde D}\circ \pi^{-1}=O(v^2)$,
and
$\Psi_2$ is real analytic as a map from
${\cal O}_\delta(h^{m^\prime})$ to $h^{m^\prime+1}$.
Furthermore, since
$\underline{\Psi}=\pi\underline{F^{\tilde D}}\circ\pi^{-1}$,
$\underline{\Psi_2}=\pi\underline{F^{\tilde D}_2}\circ\pi^{-1}$,
then
$\Psi$ and $\Psi_2$ are n.a. In addition, one has
\begin{lemma}\label{l4.6}
$\Psi_2\in\A_{m^\prime,1}$.
\end{lemma}
\begin{proof}
We already know that
$\Psi_2:{\cal O}_\delta(h^{m^\prime})\rightarrow h^{m^\prime+1}$
is n.a. and $\Psi_2(v)=O(v^2)$.
Moreover, since
$d\Psi_2(v)^*g=\pi dF_2^{\tilde D}(\pi^{-1}v)^*\pi^{-1}g$,
where $dF_2^{\tilde D}(\pi^{-1}v)^*=\overline{dF_2^{\tilde D}(\pi^{-1}v)^t}$,
the map $v\mapsto d\Psi_2(v)^*$,
${\cal O}_\delta(h^{m^\prime})
\rightarrow {\cal L}(h^{m^\prime},h^{m^\prime+1})$ is
real analytic  by Lemma \ref{l4.5}.
Finally, the representation
$\underline{d\Psi_2(|v|)^*|f|}=\Phi(|v|)|f|$, where
$\Phi:{\cal O}_\delta(h_R^{m^\prime})
\rightarrow {\cal L}(h^{m^\prime}_R,h^{m^\prime+1}_R)$ is
real analytic and  $\big(\Phi(v)\big)_{jk}\geq 0$
for $v=|v|$, required by the definition of $\A_{m^\prime,1}$,
follows from the identity
$$\underline{d\Psi_2(|v|)^*|f|}=\pi
d\underline{F_2^{\tilde D}}(\pi^{-1}|v|)^t\pi^{-1}|f|$$
and item (ii) of Lemma \ref{l4.5}.
\end{proof}
To finish the proof of Theorem \ref{t3.1} it remains to
note that assertion ii)
follows from Lemma \ref{l4.1} ii). Indeed, if $v=\tilde D^{-1} {\cal F}(u)$,
then
\begin{equation*}
\big|\Psi^j(v)\big|^2=j^{-1}|z_j(u)|^2=
\pi|j|^{-1}(\la_{2j}(u)-\la_{2j-1}(u))^2=\pi|j|^{-1}\gamma_j^2(u).
\end{equation*}
This concludes the proof of Theorem \ref{t3.1}.
\par It remains to prove Lemma \ref{l4.2}.
We will obtain it as a consequence of
\begin{lemma}\label{l4.7}
$Z_n^j(w)$ (see \eqref{4.7}) can be represented as
\begin{equation}\label{4.90}
Z_n^j(w)=\sum\limits_{{i}=(i_1,i_2,\dots, i_n)\in \Z_0^n}
\tilde{\cal K}_n^j({i})w_{i_1}w_{i_2}\dots w_{i_n},
\end{equation}
where for $n\ge2$
$\tilde {\cal K}_n^j$ satisfies
\begin{equation}\label{4.16}
\begin{aligned}
\big|\tilde {\cal K}_n^j(i)\big|&\leq R^n{\cal A}_n^j(i),\quad
{\cal A}_n^j(i)=\delta_{j,\sum_{l=1}^ni_l}\,a_n^j(i_1,\dots,i_{n-1}),\\
a_n^j(i_1,\dots, i_{n-1})=&
\prod\limits_{l=1}^{n-1}
\Big(\left<\sum_{k=1}^l i_k\right>\left<\sum_{k=1}^l i_k-j\right>\Big)^{-1}
\end{aligned}
\end{equation}
Here $<j>=(1+j^2)^{1/2}$.
\end{lemma}
\noindent
{\it Remark.} The difference between representation \eqref{4.9}
and \eqref{4.90} is that $\tilde {\cal K}_n^j(i)$ are
not required  to be symmetric functions.
\par Clearly, item (i) of Lemma \ref{l4.2} follows trivially from Lemma
\ref{l4.7}.
It is also not difficult to check that
\eqref{4.16} leads to the following estimates
$$\|\tilde{\cal K}_n^j\|_{l^2(\Z_0^n)}\leq
\frac{R^n}{|j|^{n-1}},\quad n\geq 2,\quad
\sup\limits_{1\leq l\leq n}
\|{\cal B}_{n,l}^j\|_{l^2(\Z_0^n)}\leq \frac{R^n}{|j|^{2}},\quad n\geq 3,$$
$${\cal B}_{n,l}^j(i_1,\dots, i_n)=
\tilde{\cal K}_n^{i_l}(i_1,\dots,i_{l-1},j, i_{l+1},\dots,i_n),$$
which in turn imply (ii), (iii) of Lemma \ref{l4.2}.
So it remains to establish Lemma~\ref{l4.7}.
\par {\it Proof of Lemma \ref{l4.7}}.
First notice that for $n=2,3$ the representation
\eqref{4.90}, \eqref{4.16} follows directly from the explicit formulas
\eqref{4.12}, \eqref{4.100}. The general case can be treated as follows.
Consider $z_j(u)$ (see \eqref{4.5}) and write it as the sum
$z_j(u)=z_{j,1}(u)+z_{j,2}(u)$, where
\begin{align}
z_{j,1}(u)=&-\sqrt{\pi}\big((L_0-j^2/4)f_j(u),\overline{f_j(u)}\big),
\label{4.501}\\
z_{j,2}(u)=&\sqrt{\pi}\big(uf_j(u),\overline{f_j(u)}\big),\label{4.51}
\end{align}
The functions $f_j(u)$, $j\in \Z_0$,  were defined in \eqref{4.50}.
Now it is convenient for us to write them as
\begin{equation}\label{4.52}
f_j(u)=\big(I-P_{|j|1}^2\big)^{-1/2}(I+P_{|j|1})f_{j0}, \quad
P_{|j|1}(u)=P_{|j|}(u)-P_{|j|0}.
\end{equation}
The Taylor expansions for $z_{j,k}(u)$, $k=1,2,$
have the form
$$z_{j,1}(u)=\sum\limits_{n\geq2}Z_{n}^{j,1}(w),\quad
z_{j,2}(u)=\sum\limits_{n\geq1}Z_{n}^{j,2}(w),\quad w={\cal F}(u),$$
where $Z_n^{j,k}$ are bounded $n$-homogeneous functionals on $h^0(\Z_0)$,
and $Z_n^j=Z_n^{j,1}+Z_n^{j,2}$.

 We next  compute explicitly  $Z_n^{j,k}(w)$.
From \eqref{4.52} we have
\begin{equation}\label{4.53}
f_j(u)=\sum\limits_{q\geq 0}c_qP^q_{|j|1}(u)f_{j0},
\end{equation}
where $c_q={\varphi^{(q)}(0)}/{q!}$ for  $\varphi(x)=(1-x^2)^{-1/2}(1+x)$.
Note that $c_q\geq 0$.
\par Further expanding $P_{|j|1}(u)$ in the power series  of $u$:
$$P_{|j|1}(u)=-\frac{1}{2\pi i}\sum\limits_{k\geq 1}
\oint\limits_{\gamma_{|j|}}T^k(u,\la)\big(L_0-\la\big)^{-1}d\la,
\quad T(u,\la)=(L_0-\la)^{-1}u,$$
and substituting this expansion into \eqref{4.53}
we get
$$f_j(u)=f_{j0}+\sum\limits_{n\geq 1}
\sum\limits_{1\leq q\leq n}c_q
\sum\limits_{{\alpha=(\a_1,\dots\a_q)\in\N^q,\atop |\a|=n}}f_{j,q}^\a(u),$$
$$f_{j,q}^\a(u)=$$
$$\left(\frac{i}{2\pi}\right)^q\oint\limits_{\gamma_{|j|}}
\dots\oint\limits_{\gamma_{|j|}}
T^{\a_1}(u,\la_1)(L_0-\la_1)^{-1}
\dots
T^{\a_q}(u,\la_q)(L_0-\la_q)^{-1}f_{j0}d\la_1\dots d\la_q.$$
Substituting this series into \eqref{4.501}
and replacing $u$ by
$\frac{1}{2\sqrt{\pi}}\sum\limits_{j\in \Z_0}w_je^{ijx}$,
we arrive at the following representation for $Z_n^{j,1}$:
\begin{equation*}\begin{split}
Z_n^{j,1}(w)&=\sum\limits_{{i}=(i_1,i_2,\dots, i_n)\in \Z_0^n}
\tilde{\cal K}_n^{j,1}(i)w_{i_1}w_{i_2}\dots w_{i_n},\\
\tilde {\cal K}_n^{j,1}(i)&=
\sum\limits_{{ p=(p_1,p_2)\in\N^2,\atop |p|\leq n}}c_{p_1}c_{p_2}
\sum\limits_{{\beta=(\b_1,\dots,\b_{| p|})\in\N^{|p|},\atop|\b|=n}}
{\cal S}^{j,1}_{p,\b}( i),\quad n\geq 2.
\end{split}
\end{equation*}
Here
\begin{equation*}\begin{split}
&{\cal S}^{j,1}_{p,\b}( i)=\\
&=-
\left(\frac{1}{2\sqrt{\pi}}\right)^{n-1}
\left(\frac{i}{2\pi}\right)^{| p|}
\oint\limits_{\gamma_{|j|}}\dots\oint\limits_{\gamma_{|j|}}
s_{ p,\b}^{j,1}(i,\la)d\la_1\dots d\la_{|p|},
\quad \la=(\la_1,\dots,\la_{|p|})
\end{split}
\end{equation*}
and
\begin{equation}\label{4.60}
\begin{split}
&s_{p,\b}^{j,1}(i,\la)=\\
&=\delta_{j,\sum\limits_{l=1}^n i_l}
\left(\prod\limits_{m=1}^n\frac{1}{\big(\sum_{l=1}^m i_l-j/2\big)^2-\mu_m}
\right)
\left(\prod\limits_{m=1}^{|p|-1}
\frac{1}{\big(\sum_{l=1}^{\b_1+\dots+\b_m}i_l-j/2\big)^2-\la_{m+1}}\right)\\
&\quad\times\frac{\big(\sum_{l=1}^{\b_1+\dots+\b_{p_1}}i_l-j/2\big)^2-j^2/4}
{j^2/4-\la_1}
\end{split}
\end{equation}
with some $\mu_m=\mu_m(\la;p,\b)\in\gamma_{|j|}$.
\par Since for any $r$, we have $|\la_r-j^2/4|,\,|\mu_r-j^2/4|=\delta_0|j|$,
where
$\delta_0>0$ small, one deduces from \eqref{4.60} that
\begin{equation*}\label{4.29}
|s_{p,\b}^{j,1}(i,\la)|\leq
\frac{C^n}{<j>^{|p|}}{\cal A}_n^j(i).
\end{equation*}
As a consequence, one gets
$\big |{\cal S}^{j,1}_{p,\b}(i)\big|\leq C^n{\cal A}_n^j(i)$,
so that
$$\big |\tilde {\cal K}_n^{j,1}(i)\big|\leq C^n{\cal A}_n^j( i)
\sum\limits_{{p=(p_1,p_2)\in\N^2,\atop |p|\leq n}}c_{p_1}c_{p_2}
\sum\limits_{{\beta=(\b_1,\dots,\b_{|p|})\in\N^{|p|},\atop|\b|=n}}1
\leq C^n{\cal A}_n^j(i).$$
Similar computations can be performed for $Z_n^{j,2}$, $n\geq 2$.
As a result, one gets the representation
$$Z_n^{j,2}(w)=\sum\limits_{{i}=(i_1,i_2,\dots, i_n)\in \Z_0^n}
\tilde{\cal K}_n^{j,2}({i})w_{i_1}w_{i_2}\dots w_{i_n},$$
with some $\tilde{\cal K}_n^{j,2}$, satisfying the same estimate as
 $\tilde{\cal K}_n^{j,1}$:
$\ \big |\tilde {\cal K}_n^{j,2}(i)\big|\leq C^n{\cal A}_n^j(i).$
This concludes the proof of Lemma \ref{l4.7}. ~~~~~~~~~~~~~~~~~~~~~~~~$\square$
\medskip

\noindent {\it Remark}. The map $\Psi$ preserves the form $\omega_0$ up to  terms of order
$v^2$:
\begin{equation*}
\Psi^*\omega_0(v)=J(v)dv\wedge dv,\quad J(v)=i+O(v^2).
\end{equation*}
Indeed, it follows from \eqref{4.12} that
$$\Psi(v)=v+\psi_2(v)+O(v^3),$$
where
$\psi_2=(\psi_2^j,\,j\in\N)$ is given by
\begin{equation}\label{5.20}
\psi_2^j(v)=\frac{1}{2\sqrt{\pi j}}\sum\limits_{k\in\Z_0,\,k\neq j}
\frac{|k|^{1/2}|j-k|^{1/2}}{k(k-j)}v^\prime_kv^\prime_{j-k}
\end{equation}
$$
v^\prime_j=v_j \,\,{\rm for}\,\, j\geq 1, \,\,v^\prime_j=\bar v_{-j}\,\,
\rm{for}\,\, j\leq -1.$$
Computing $\Psi^*\omega_0$ we get
$$\Psi^*\omega_0(v)=(i+O(v^2))dv\wedge dv+d\left(\frac{i}{2}\psi_2(v)dv\right).$$
Note that for $\forall\,\, j,k\in\N$
$$\frac{\p \psi_2^j}{\p\bar v_k}=\frac{\p \psi_2^k}{\p\bar v_j},\qquad
\frac{\p \psi_2^j}{\p v_k}=-\frac{\p \bar \psi_2^k}{ \p \bar v_j}\,.$$
Therefore, $d(i\psi_2(v)dv)=0$ and $\Psi^*\omega_0(v)=(i+O(v^2))dv\wedge dv$.

\bibliography{vey}

\bibliographystyle{amsalpha}

\end{document}